\documentclass [11pt] {article}
\usepackage{amsfonts}
\usepackage{amssymb}
\usepackage{times}
\usepackage{graphicx}
\usepackage[all]{xy}

\newtheorem{lemma}{Lemma}[section]
\newtheorem{corollary}{Corollary}[section]
\newtheorem{theorem}{Theorem}

\newtheorem{remark}{Remark}[section]

\newtheorem{proposition}{Proposition}[section]
\newtheorem{definition}{Definition}[section]

\def\tr{\mbox{Tr}\,}

\def\limo{\lim_{\omega}}

\def\dim{{\rm dim}}

\def\<{\langle}
\def\>{\rangle}

\def\bA{{\bf A}}

\def\conv{{\rm Conv}}
\def\aut{{\rm Aut}}
\def\stab{{\rm Stab}}
\begin{document}

\title{Higher order Fourier analysis as an algebraic theory III.}
\author{{\sc Bal\'azs Szegedy}}

\maketitle

\abstract{ For every natural number $k$ we introduce the notion of $k$-th order convolution of functions on abelian groups. We study the group of convolution preserving automorphisms of function algebras in the limit. It turns out that such groups have $k$-nilpotent factor groups explaining why $k$-th order Fourier analysis has non-commutative features. To demonstrate our method in the quadratic case we develop a new quadratic representation theory on finite abelian groups.
We introduce the notion of a quadrtic nil-morphism of an abelian group into a two step nil-manifold. We prove a structure theorem saying that any bounded function on a finite abelian group is decomposable into a structured part (which is the composition of a nil-morphism with a bounded complexity continuous function) and a random looking part with small $U_3$ norm. It implies a new inverse theorem for the $U_3$ norm. (The general case for $U_n$, $n\geq 4$ will be discussed in the next part of this sequence.) We point out that our framework creates interesting limit objects for functions on finite (or compact) abelian groups that are measurable functions on nil-manifolds.
}

\tableofcontents

\section{introduction}

In this paper we continue our program started in \cite{Sz1} and \cite{Sz2} to develop an algebraic framework for higher order Fourier analysis using limits of functions on abelian groups.
Higher order characters and decompositions were studied in the first two papers. In this paper we set up the framework which is needed to understand the structure of higher order characters. Our main goal is to explain, probably the most fascinating fact, why non-commutative structures, such as $k$-nilpotent groups, arise when studying the $U_{k+1}$ Gowers norm on abelian groups.
The key ingredient is the notion of {\bf $k$-th order convolution} defined as the expected value of

$$\prod_{S\subseteq [k]}f_S\Bigl(x+\sum_{i\in S}t_i\Bigl)$$
as $x$ is taken randomly and $\{f_S\}_{S\subseteq [k]}$ is a system of $2^k$ functions on the abelian group $A$. The $k$-th order convolution is a $k$ variable function. For $k=1$ it is equal to the usual convolution.

If $A_1,A_2,\dots$ is a sequence of growing abelian groups then their ultra product $\bA$ has a natural probability space structure with shift invariant measure. For a shift invariant $\sigma$-algebra $\mathcal{B}$ on $\bA$ we study the subgroup $G\subset\aut(L_\infty(\mathcal{B}))$ consisting of those automorphisms that preserve $k$-th order convolution for every $k$. It turns out that if $\mathcal{B}$ is embedded into $\mathcal{F}_k$ (a canonical $\sigma$-algebra on $\bA$ related to $k$-th order Fourier analysis) then $G$ is a $k$-nilpotent group. Furthermore there is a natural way of representing the functions in $L_\infty(\mathcal{B})$ on this group.

In the first chapter of this paper we prove several general results about the structure of $G$ for every $k$. Note that this part uses a lots of intuition from the fundamental paper \cite{HKr} by Host and Kra.
In the second chapter we deal with $k=1$ and $k=2$.
We demonstrate the efficiency of the infinite method by constructing a ``quadratic representation theory'' for finite abelian groups.
Quite surprisingly it turns out that there is a rigid algebraic notion that we call a ``quadratic nil-morphism'' mapping an abelian group into a two step nil-manifold.

Ordinary Fourier analysis can be described through homomorphisms of an abelian group into the circle. However, it turns out, that in the higher order case, the natural algebraic objects are functions (called nil-morphisms) that take values in higher dimensional compact topological spaces, so called nil-manifolds.
It makes it more interesting that nil-manifolds are not groups but only homogeneous spaces of groups.

\medskip

Let us start with a simplified example for a quadratic nil-morphism which we call {\bf orbit representation}. Let $N$ be a two nilpotent Lie group with a discrete co-compact subgroup $T$ and let $Z$ be the center of $N$. Let $M$ denote the left coset space of $T$ in $N$. The space $M$ is called a two step ``nil-manifold''. There is a simple way of representing a finite abelian group in $M$.
Assume that some subgroup $H\in N$ acts transitively on a finite point set $S\subset C$ in the same way as $A$ acts on some finite set. Using that $A$ is abelian, we get that $H/(H\cap T)$ is a homomorphic image of $A$ and so there is an epimorphism $\phi:A\rightarrow H/(H\cap T)$ which also defines a map $\psi:A\rightarrow M$.
We say that $\psi$ is an orbit representation. In other words, an orbit representation is a map $\psi:A\rightarrow M$ such that for every $a\in A$ there is an element $n_a\in N$ such that
$$\psi(a+b)=n_a\psi(b)$$
for every $b\in A$.

This will not be quite enough for our purposes. We will need twisted versions of orbit representations. (Note that in the higher order case a more complicated definition will be needed.)

\begin{definition} We say that a map $\psi:A\rightarrow M$ is a {\bf quadratic nil-morphism} if for every $a\in A$ there is an element $n_a\in N$ and a homomorphism $\chi_a:A\rightarrow Z$ such that
$$\psi(a+b)=\chi_a(b)n_a\psi(b)$$
for every $b\in A$.
\end{definition}

The good news is that quadratic nil-morphisms can be understood through ordinary homomorphisms if we consider an enlarged version of $N$ (A certain group that contains $N$ as a section) that we usually call an {\bf interpretation} of $N$. However this enlargement depends on the group $A$ which is mapped into $M$.
For a function $f:A\rightarrow \mathbb{C}$ let $\Delta_t f$ denote the function $x\mapsto f(x+t)\overline{f(x)}$. A basic example for a quadratic nil-morphism is a function $\psi:A\rightarrow\mathbb{C}$ such that
$$\Delta_{t_1,t_2,t_3}\psi (x)=1$$ for every $t_1,t_2,t_3,x\in A$.
In this case $N=Z$ is the complex unit circle and $T$ is trivial.
Such functions can also be described in terms of ordinary homomorphisms $\alpha$ of $A$ into the semi direct product $G=A\ltimes(\mathcal{C}\times\hat{A})$ where $\mathcal{C}$ is the circle, $(c,\chi)^a=(c\chi(a),\chi)$ and $\alpha$ composed with $G\rightarrow A$ is the identity map on $A$.

\medskip

The notion of a polynomial map was introduced by Leibman \cite{Leib}.

\begin{definition}[Polynomial] A mapping $\phi$ of a group $G$ to a group $F$ is said to be plynomial of degree $k$ it trivializes after $k+1$ cosecutive applications of operators $D_h$,$h\in G$ defined by $D_h\phi(g)=\phi(g)^{-1}\phi(gh)$.
\end{definition}

Now we can define nil-morphisms in general.

\begin{definition}[Nil-morphism] Let $N$ be a $k$-nilpotent group, $T$ be a subgroup of $N$ and $M$ be the left coset space of $T$ in $N$. The factor map $g:N\rightarrow M$ is defined by $g(t)=tN$. A map $\phi$ from a group $A$ to $M$ is called a degree $k$ nil-morphism if there is a group $B$, surjective homomorphism $f:B\rightarrow A$ and polynomial map $\phi':B\rightarrow N$ of degree (at most) $k$ such that the next diagram is commutative.

\begin{displaymath}
    \xymatrix{
        B \ar[r]^{\phi'} \ar[d]_{f} & N \ar[d]^{g} \\
        A \ar[r]_{\phi}       & M }
\end{displaymath}
Let $N=N_0\supset N_1\supset\dots\supset N_k=\{1\}$ be the upper central series of $N$. In most applications we will also assume that $\phi':B\rightarrow N/N_i$ is a polynomial of degree $i$.
\end{definition}

Our program is to build up Higher order Fourier analysis in terms of higher order nil-morphisms.
The main ingredient of this program is to prove a structure theorem for functions on abelian groups which is analogous to Szemer\'edi's regularity lemma. It says that for any bounded function $f:A\rightarrow\mathbb{C}$ there is a nil-morphism $\psi:A\rightarrow M$ into a bounded dimensional $k$ step nil-manifold and a ``bounded complexity function'' $\phi:M\rightarrow\mathbb{C}$ such that
$$f(x)=h(x)+g(x)+\phi(\psi(x))$$
where $\|h\|_2$ is small and $\|g\|_{U_{k+1}}$ is arbitrary small in terms of the complexity of $\phi$ and $M$.
We will deal with the general case in the next paper.

Let us focus on the quadratic case. There are many possible ways of defining a bounded complexity function on $M$. A very natural choice is to prescribe a continuous embeddings $v:M\hookrightarrow\mathbb{R}^n$ for every nil-manifold and than say that $\phi:M\rightarrow\mathbb{C}$ has complexity at most $d$ if there is a polynomial function $p:\mathbb{R}^n\rightarrow\mathbb{C}$ of degree at most $d$ such that $\phi(x)=p(v(x))$.
Out of technical reasons we also assume that the the embedding of direct products of nil-manifolds is the direct product of the embeddings. Let us furthermore define the complexity of a manifold as the sum of its dimension and the number of connected components.

\begin{theorem}[Quadratic structure theorem] For an arbitrary function $F:\mathbb{R}^+\times\mathbb{N}\rightarrow\mathbb{R}^+$ and every $\epsilon>0$ there is a natural number $n$ and a two step nil-manifold $M$ of complexity at most $n$ such that for every function $f:A\rightarrow\mathbb{C}$ on a finite abelian group $A$ with $\|f\|_\infty\leq 1$ there is a decomposition $f=h+g+\psi\circ\phi$ where
\begin{enumerate}
\item $\psi:A\rightarrow M$ is a quadratic nil-morphism,
\item $\phi:M\rightarrow\mathbb{C}$ is a function of complexity at most $n$,
\item $\|g\|_{U_3}\leq F(\epsilon,n)$,~ $\|g\|_\infty\leq 1$,
\item $\|h\|_2\leq\epsilon$,
\item $|(h,g)|~,~|(h,\psi\circ\phi)|~,~|(g,\psi\circ\phi)|~\leq F(\epsilon,n)$.
\end{enumerate}
\end{theorem}

\bigskip

\begin{remark}[Uniform version] Without proof we mention that the structure theorem remains true if we require that the range of $\psi$ is $F(\epsilon,n)$ close to the uniform distribution in $M$ (with respect to the Haar measure) in an arbitrary (but prescribed) metric on the space of probability distributions. However in this case $M$ will depend on the choice of $\epsilon$. This uniform version of the structure theorem has the big advantage that continuous embeddings of $M$ can be replaced by "almost continuous" embeddings (that are continuous on the complement of a closed $0$ measure set). There are nice embeddings of this type.
\end{remark}

\begin{remark} We will see, the structure of $M$ can be restricted if $A$ has special structure. Foe example if $A$ has bounded exponent then $M$ can be assumed to be $0$ dimensional.
\end{remark}

To formulate inverse theorems for the $U_3$ norm of conventional type we need to specify what we mean by correlation with a nil-morphism. This can be done in many different ways. Most of them are equivalent up to changing constants. We devote a chapter to this issue.
The simplest definition is to say that $f$ {\bf $\delta$-correlates} with a map $\psi:A\rightarrow M$ if $$(f,m(v(\psi(x))))>\delta$$ for some arbitrary but fixed continuous embedding $v:M\hookrightarrow\mathbb{R}^n$ and monomial
$$m(x_1,x_2,\dots,x_n)=\prod_{i=1}^n x_i^{t_i}$$ of degree at most $1/\delta$.
Note that we normalize scalar products by the size of $A$

\medskip

Now we are ready to formulate the inverse theorem.

\begin{theorem}[Quadratic inverse theorem] For every $\epsilon>0$ there is a nil-manifold $M$ and $\delta>0$ such that for every function $f:A\rightarrow\mathbb{C}$ with $\|f\|_{U_3}\geq\epsilon$ and $\|f\|_\infty\leq 1$ there is nil-morphism $\psi:A\rightarrow M$ such that $f$ $\delta$-correlates with $\psi$.
\end{theorem}

Note that Green and Tao \cite{GrTao} proved an inverse theorem for functions on cyclic groups which has a similar flavor but it used nil-sequences instead of nil-morphisms.

This inverse theorem with the above choice of complexity notion and correlation notion is obviously a consequence of the quadratic structure theorem.
It is interesting to mention that polynomial functions and monomials are arbitrary choices for ``dense enough'' function systems on nil-manifolds but in some sense they seem to be quite natural. Let us consider the natural embedding $x\rightarrow (\cos(x2\pi),\sin(x2\pi))$ of the circle group $\mathbb{R}/\mathbb{Z}$ into $\mathbb{R}^2$. A Polynomial function with respect to this embedding composed with a homomorphism $\phi:A\rightarrow\mathbb{R}/\mathbb{Z}$ can easily be expressed in terms of finitely many linear characters. However it is not clear which embedding of a general nil-manifold is the most natural.

\medskip

As a byproduct of our framework, an {\bf interesting limit notion for functions on abelian groups} can be defined.
The limit notion depends on a natural number $k$ which measures how much information we want to preserve in the limit. For $k=1$ the limit notion is only sensitive to dominant fourier coefficients. Quite interestingly for a sequence of measurable functions on compact abelian groups the limit object is again a measurable function on some compact abelian group. It can happen surprisingly, that limit of functions on the circle group is only defined on the torus.
For $k=2$ the limit object is a measurable function on the inverse limits of two step nil-manifolds.
This limit notion comes up in this paper naturally however a much more detailed description will be given in another paper.

\noindent

Finally we give a rough sketch of our method to shows where nilpotent Lie groups appear.

\bigskip

\noindent{\bf Step 1.~(getting a two nilpotent group)} Let $f_1,f_2,\dots$ be a sequence of bounded functions on a sequence of growing abelian groups $A_1,A_2,\dots$. Let $f$ be the ultra limit function on $\bA$ and let
$f=h+g_1+g_2+\dots$ be the quadratic fourier decomposition defined in \cite{Sz1} and \cite{Sz2}.
Let us take a component, say $g_1$, and take the smallest shift invariant $\sigma$-algebra $\mathcal{B}$ on $\bA$ containing $g_1$. Let $M$ be the function algebra $L_\infty(\mathcal{B})$ and let $G(M)$ be the group of convolution preserving automorphisms of $M$.
It follows from the first part of our paper that $G(M)$ is a two nilpotent group.
However $G(M)$ is very big. We need to find a Hausdorff topological group somewhere in it.

\medskip

\noindent{\bf Step 2.~(extracting a nice topological group)} Let $S$ be the stabilizer of $g_1$ in $G(M)$ and $H$ be the centralizer of $S$. Then it turns out that $N=H/S$ is a locally compact, Hausdorff topological group.
Furthermore $N$ has the very specific structure
$$0\rightarrow T\times Z\rightarrow N\rightarrow \hat{T}\rightarrow 0$$ where $Z$ is a closed subgroup in the circle group, $T$ is a discrete abelian group and the action of $\hat{T}$ is $T$ is defined by $(a,b)^c=(a,c(a)b)$. Furthermore $Z$ contains the commutator subgroup.
We call such extensions nil-patterns of type $T$. Note that such extensions can be described by elements of the second co-homology group $H^2(\hat{T},T\times Z)$. It can be seen that any homomorphism $\alpha:T\rightarrow T_2$ induces a map $H^2(\hat{T},T\times Z)\rightarrow (\hat{T_2},T_2\times Z)$ and so it induces a new nil pattern $N_2$ of type $T_2$ that we call the interpretation of $N$ under $\alpha$. There is a reverse map $N_2/T_2\rightarrow N/T$ between the corresponding nil-manifolds.

\noindent{\bf Step 3.~(nil-morphisms appear in the limit)} We show that there is a natural nil-morphism $\Psi:\bA\rightarrow N/T$ and a measurable function $s:N/T\rightarrow\mathbb{C}$ such that
$$g_1(x)=\Psi(s(x)).$$ In other words the function $g_1$ can be lifted to $N/T$ in a natural way.

\noindent{\bf Step 4.~(getting a finite dimensional Lie group) } In this step we extract a finite dimensional Lie group from $N$. We prove a theorem saying that $N$ is the interpretation of a nil pattern $N_2$ of type $T_2$ under a monomorphism $T_2\rightarrow T$ such that $T_2$ is finitely generated.
This means that $\hat{T_2}$ is a finite extension of a torus and so $N_2/T_2$ is finite dimensional. We deduce that $g_1$ can be approximated in the form $h(\Psi_2)$ where $\Psi_2$ is a nil-morphism into a finite dimensional nil manifold and $h$ is a measurable function on the nil-manifold.
Note that the group $T_2$ appears as a subgroup of the ultra product group $\hat{\bA}$
which forces restrictions on $T_2$ in terms of the structures of $\{A_i\}_{i=1}^\infty$. For example if $A_i$ are of bounded exponent then so is $T_2$ and $\hat{T}$. Consequently $N_2$ is a finite extension of the circle group.

\noindent{\bf Step 5.~(going back to finite groups)} We show that a nil-morphism $\Psi:\bA\rightarrow N/T$ into a nil-pattern is the ultra limit of nil-morphisms $\Psi_i:A_i\rightarrow N/T$. Then after approximating the above function $h$ by polynomials we can pull these polynomials back to the finite abelian groups $A_i$ using the maps $\Psi_i$.

\section{Higher order convolutions and nilpotent groups}

\subsection{The $k$-th order convolution}

\begin{definition}[$k$-th order convolution] Let $A$ be an abelian group with a shift invariant probability measure $\mu$. Let $F=\{f_S~|~S\subseteq [k]\}$ be a collection of $2^k$ functions in $L_\infty(A,\mu)$. We define the $k$-th order convolution $\conv_k(F)$ as the $k$ variable function
$$\conv_k(F)(t_1,t_2,\dots,t_k)=\int_x \prod_{S\subseteq [k]}f_S\Bigl(x+\sum_{i\in S}t_i\Bigl).$$
\end{definition}

The $k$-th order convolution has the next crucial property.

\begin{lemma}\label{convl2norm} $$\|\conv_k(F)\|_2\leq\prod_{S\subseteq[k]}\|f_S\|_{U_{k+1}}^{2^{k+1}}.$$
\end{lemma}

\begin{proof} Let us consider the expression
$$Q=\mathbb{E}_{x,t_1,t_2,\dots,t_{k+1}}\Bigl(\prod_{S\subseteq[k]}f_S(x+\sum_{i\in S}t_i)\overline{f_S}(x+t_{k+1}+\sum_{i\in S}t_i)\Bigr).$$
It is clear that $Q=\|\conv_k(F)\|_2^2$. On the other hand by the so called Gowers-Cauchy-Schwartz inequality we have that $Q\leq\prod_{S\in[k]}\|f_S\|_{U_{k+1}}^{2^{k+2}}$.
\end{proof}

\begin{remark} The previous lemma implies that if we change a function in the system $F$ by a function that has a small $U_{k+1}$ norm then the function $\conv_k(F)$ changes only a little in $L_2$. In other words the $k$-th order convolution is continuous in the $U_{k+1}$ norm. As the next corollary says, this shows the surprising thing that the $k$-th order convolution depends only on the $k$-th degree structure of the functions in $F$.
\end{remark}

\begin{corollary} Let $F=\{f_S~|~S\subseteq [k]\}$ be a collection of $L_\infty$ functions on an ultra product group $\bA$ and let $F'$ be the system $\{f_S'=\mathbb{E}(f_S|\mathcal{F}_k)~|~S\subseteq [k]\}$. Then $\conv_k(F)=\conv_k(F')$.
\end{corollary}

\begin{proof} It is trivial since $\|f_S-f_S'\|_{U_{k+1}}=0$.
\end{proof}

\begin{lemma}\label{l2int} For a collection of $2^k$ functions $F=\{f_S~|~S\subseteq [k]\}$ and for an arbitrary number $1\leq i\leq k$ we have that
$$\|\conv_k(F)\|^2_2\leq \int_t \prod_{S\subseteq [k]\setminus\{i\}}\|f_S(x)f_{S\cup\{i\}}(x+t)\|_{U_k}^{2^{k+1}}$$
\end{lemma}

\begin{proof} For every fixed $t$ we define the system
$$F_{t}=\{f_S(x)f_{S\cup\{i\}}(x+t)|S\subseteq [k]\setminus\{i\}\}.$$
By lemma \ref{convl2norm} we have that
$$\|\conv_{k-1}(F_t)\|_2\leq\prod_{S\subseteq [k]\setminus\{i\}}\|f_S(x)f_{S\cup\{i\}}(x+t)\|_{U_k}^{2^k}.$$
On the other hand we have that
$$\|\conv_k(F)\|^2_2=\int_t\|\conv_{k-1}(F_t)\|_2^2$$
and so the proof is complete.
\end{proof}

\begin{definition} Let $F=\{f_S~|~S\subseteq [k]\}$ be a system of $L_\infty$ functions. The so called Gowers inner product $(F)_k$ is defined by
\begin{equation}\label{gowinner}
(F)_k=\mathbb{E}_{x,t_1,t_2,\dots,t_k}\Bigl(\prod_{S\subseteq[k]}f_S^{c(|S|)}(x+\sum_{i\in S}t_i)\Bigr)
\end{equation}
where $c(n)$ is the complex conjugation applied $n$-times.
\end{definition}

 The next lemma shows that the Gowers inner product $(F)_{k+1}$ of $2^{k+1}$ functions is the usual inner product of two $k$-th order convolutions.

\begin{lemma} Assume that $[k+1]=K\cup\{i\}$ then
\begin{equation}\label{convin}(F)_{k+1}=\Bigl(\conv_k(\{f_S^{c(|S|)}|S\subseteq K\}),\conv_k(\{f_{S\cup\{i\}}^{c(|S|)}|S\subseteq K\})\Bigr).
\end{equation}
\end{lemma}

\begin{proof}
In the formula (\ref{gowinner}) we introduce the new variable $y=x+t_i$. It is clear that the expected value taken for independent values of $x,t_1,t_2,\dots,t_{i-1},y,t_{i+1},\dots,t_{k+1}$ remains the same. This completes the proof.
\end{proof}

\subsection{Convolution on ultra product functions}

Let $\bA$ be an ultra product group. In this part we study convolutions on $\bA$.

\begin{lemma}\label{convnull} Let $F$ be a system of $2^k$ functions in $L_\infty(\mathcal{F}_k)$ such that each of them is contained in a rank one module over $L_\infty(\mathcal{F}_{k-1})$. Assume furthermore that for some $i$ there is a set $S\subseteq [k]$ not containing $i$ such that $f_Sf_{S\cup\{i\}}$ is not in the trivial module. Then $\conv_k(F)=0$.
\end{lemma}

\begin{proof} Under the condition of the lemma there is a nontrivial module containing $f_S(x)f_{S\cup\{i\}}(x)$ and so the same module contains $f_S(x)f_{S\cup\{i\}}(x+t)$ for every $t$. It implies that $\|f_S(x)f_{S\cup\{i\}}(x+t)\|_{U_k}=0$ for every $t$. Then lemma \ref{l2int} show that $\|\conv_k(F)\|_2=0$ which completes the proof.
\end{proof}

\begin{corollary}[Additivity]\label{addconv} Let $\bA$ be an ultra product group and let $F$ be a system of $2^k$ functions in $L_\infty(\mathcal{F}_k)$. For every $c\in\hat{\bA}_k$ let $F_c$ denote the system
$\{\mathbb{E}(f_S|c^{\epsilon(|S|)})|S\in [k]\}$ where $\epsilon(n)=(-1)^n$.
Then
$$\conv_k(F)=\sum_{c\in \hat{\bA}_k}\conv_k(F_c).$$
\end{corollary}

\subsection{Automorphisms of function algebras}

Let $(X,{A},\mu)$ be a measure space and let $M$ be that algebra of complex valued functions in $L_\infty(X)$ (up to $0$ measure change) with respect to point wise multiplication.
The algebra $M$ is a $*$ algebra with respect to point wise conjugation and has a trace operation defined by $\tr(f)=\int_Xf~d\mu$.

\begin{definition} An automorphism $\sigma$ of $M$ is a transformation satisfying the following axioms.
\begin{enumerate}
\item $\sigma$ is an invertible linear transformation of the space $L_\infty(X)$,
\item $(fg)^\sigma=f^\sigma g^\sigma$,
\item $(\overline{f})^\sigma=\overline{f^\sigma}$,
\item $\tr(f^\sigma)=\tr(f)$.
\end{enumerate}
\end{definition}

The reader can easily see that automorphisms preserve the unit element. The numbers $\tr(f^i\overline{f}^j)$ determine the joint moments of the real and imaginary parts of $f$ where $x$ is chosen randomly. Since the distribution is compactly supported, this sequence determines the value distribution of $f$. It follows that automorphisms preserve the value distribution. As a consequence they preserve all possible norms depending only on the value distribution including $L_1,L_2$ and $L_\infty$.
The fact that automorphisms are trace preserving imply in particular that scalar products are also preserved. Note that automorphisms basically behave as measure preserving transformations of $X$, but for us it is more convenient to use this algebraic language.

The next lemma is a useful condition to check if $\sigma$ is an automorphism.

\begin{lemma}\label{autextends} Let $S\subseteq M$ be a set of elements of $M$ closed under multiplication, multiplication by scalars and conjugation. Assume that every element $f\in M$ can be approximated in $L_2$ by finite linear combinations of elements in $S$. Let $\sigma:S\rightarrow S$ be an invertible, multiplicative transformation satisfying~ $(\lambda s)^\sigma=\lambda s^\sigma$ for $s\in S,\lambda\in \mathbb{C}$ which commutes with conjugation and preserves the trace operation. Then $\sigma$ extends to an automorphism on $M$ in a unique way.
\end{lemma}

\begin{proof} Let $M_2$ be the algebra formed by finite sums of elements of $S$.
First we show that there is a unique linear map $\sigma_2:M'\rightarrow M'$ that extends $\sigma$. To see that $\sigma_2$ exists it is enough to check that for every finite sum $s_1+s_2+\dots+s_n=0$ with $\{s_i\}_{i=1}^n \subset S$ we have that $s_1^\sigma+s_2^\sigma+\dots+s_n^\sigma=0$.
To see this we use
$$0=\|s_1+s_2+\dots+s_n\|^2_2$$
and so
$$0=\sum_{1\leq i,j\leq n}\tr(s_is_j^*)=\sum_{1\leq i,j\leq n}\tr(s_i^\sigma (s_j^\sigma)^*)=\|s_1^\sigma+s_2^\sigma+\dots+s_n^\sigma\|_2^2.$$
The uniqueness of $\sigma_2$ follows from the fact that $S$ is a generating system.
The multiplicativity of $\sigma_2$ follows from the fact that it is multiplicative on $S$.
Note that $\sigma_2$ preserves the $\tr$ operation and commutes with the conjugation automatically. In particular $\sigma_2$ preserves $L_2$ distances.

Now let $f\in M$ be an arbitrary function, and let $f_i$ be a sequence in $M'$ converging to $f$ in $L_2$. Then $f_i^{\sigma_2}$ is a Cauchy sequence and so it has a limit that we define as $f^{\sigma_3}$. By the remarks at the beginning of this chapter, the value distribution of $f^{\sigma_3}$ has to be the same as the value distribution of $f$ and so $\|f^{\sigma_3}\|_\infty<\infty$. It follows that $f^{\sigma_3}$ is in $M$.
\end{proof}

\subsection{Convolution preserving automorphisms}

Now we start focusing on ultra product groups.
Let $\mathcal{B}$ be a shift invariant $\sigma$-algebra in $(\bA,\mathcal{F}_k)$ which is closed under projections to $L_2(\mathcal{F}_i)$ for $i=0,1,2,3,\dots$, and let $M$ denote the $*$ algebra $L_\infty(\mathcal{B})$ under the point wise multiplication with the trace operator $\tr(f)=\int f$.
We denote by $M_i$ the subalgebra $M\cap L_\infty(\mathcal{F}_i)$.

\begin{definition} We say that an automorphism $\sigma$ of $M$ preserves the $d$-th order convolution if for every system $F=\{f_S~|~S\subseteq [d]\}$ of elements in $M$ we have that
$$\conv_d(F)=\conv_d(F^\sigma)$$
where $F^\sigma=\{f_S^\sigma|S\subseteq [d]\}$.
We denote by $G_d(M)$ the group of automorphisms of $M$ that preserve the $d$-th order convolution in $M$.
\end{definition}

\begin{definition} We denote by $\iota(a)$ the shift operator on $L_\infty(A)$ defined by $f^{\iota(a)}(x)=f(x+a)$. It is trivial that $\iota(a)$ is an element in $G_d(M)$ for every $a\in A$.
\end{definition}

\begin{definition} For $\sigma\in G_k(M)$ we denote by $\Delta_\sigma$ the operator defined by $\Delta_\sigma f= f^\sigma\overline{f}$. In particular $\Delta_{\iota(t)}f=\Delta_t(f)$.
If $\sigma_1,\sigma_2,\dots,\sigma_n$ are in $G_k(M)$ then $\Delta_{\sigma_n,\sigma_{n-1},\dots,\sigma_1}f$ denote the function $\Delta_{\sigma_n}(\Delta_{\sigma_{n-1}}(\dots\Delta_{\sigma_1}(f)))\dots)$.
\end{definition}

\subsection{Cubic structure and face actions}

Following a definition by Host and Kra we introduce cubic structures and face actions.
Let $d$ be a natural number.

For a $S\subseteq [d]$ we denote the homomorphism
$\psi_S:\bA^{d+1}\rightarrow\bA$
defined by $\psi_S(x,t_1,t_2,\dots,t_d)=x+\sum_{i\in S}t_i$.

\begin{definition} An $n$ dimensional face $\Lambda=\Lambda(F,K)$ in $[d]$ is a subset of subsets of $[d]$ of the form $\{S\cup K~|~S\subseteq F\}$ where $K$ and $F$ are two fixed disjoint subsets of $[d]$ and $|F|=n$. We will denote the powerset of $[d]$ by $(d)$. Note that $(d)$ is the unique $d$-dimensional face. Zero dimensional faces are elements in $(d)$.
\end{definition}

\begin{definition}\label{facerel} For a given face $\Lambda$ we define the $\sigma$-algebra $\mathcal{B}^\Lambda$, the Banach algebra $M^\lambda$ , the set of function systems $G(M,\Lambda)$, the function $F^\Lambda$ and the semigroup $L(M,\Lambda)$ in the following way
\begin{enumerate}
\item $\mathcal{B}^\Lambda=\bigvee_{S\in\Lambda}\psi^{-1}_S(\mathcal{B})$
\item $M^\Lambda=L_\infty(\mathcal{B}^\Lambda)$
\item Let $G(M,\Lambda)$ denote the set of function systems of the form $F=\{f_S|f_S\in M,S\in\Lambda\}$
\item If  $F\in G(M,\Lambda)$ then $$F^\Lambda=\prod_{S\subseteq \Lambda}\psi_S\circ f_S.$$
\item $L(M,\Lambda)$ is the semigroup formed by all functions $F^\Lambda$ where the elements of the system $F$ are all in $M$.
\end{enumerate}
\end{definition}

\begin{remark} Let us consider the induced action of the symmetric group $S_{[d]}$ on $(d)$ together with the actions $g_K(S)=S\triangle K$ for fixed subsets $K\subseteq [d]$. They generate the wreath product of $S_d$ with the elementary abelian group $C_2^d$. We denote this group by $\aut((d))$. In other words $\aut((d))$ is the automorphism group of the $d$ dimensional cube. It is acting on the faces and on all the objects in definition \ref{facerel} related to faces.
\end{remark}

\begin{remark} The $\sigma$-algebra $\mathcal{B}^\Lambda$ is always a shift invariant $\sigma$-algebra. If $\mathcal{B}$ is contained in $\mathcal{F}_k(\bA)$ then $\mathcal{B}^\Lambda$ is contained in $\mathcal{F}_k(\bA^{d+1})$.
\end{remark}

\begin{definition}[Face actions I.] Let $\sigma\in\aut(M)$ and $\Lambda$ be a face. Then we define the action $\sigma_2=\sigma^\Lambda$ on $G(M,(d))$ by $F^{\sigma_2}=F'$ where $f'_S=f_S^\sigma$ if $S\in\Lambda$ and $f'_S=f_S$ if $S$ is not in $\Lambda$.
\end{definition}

An important lemma related to face actions is the following.

\begin{lemma}\label{gowinvf} Let $\sigma\in G_d(M)$, let $\Lambda$ be a $d$ face of $(d+1)$ and $\sigma_2=\sigma^\Lambda$. Then for $F\in G(M,(d+1))$ the Gowers inner product $(F)_{d+1}$ is invariant under $\sigma_2$, i.e. $(F)_{d+1}=(F^{\sigma_2})_{d+1}$ or equivalently $\tr(F^{(d+1)})=\tr({F^{\sigma_2}}^{(d+1)})$. Furthermore $\sigma^{(d)}$ leaves the Gowers inner product of any system $F\in G(M,(d))$ invariant.
\end{lemma}

\begin{proof} The first statement follows immediately from (\ref{convin}). The second statement follows from the fact that $(F)_d$ is is the trace of $\conv_d(F)$ which is invariant under $\sigma$.
\end{proof}

\begin{definition}[Face actions II.] Let $\sigma\in\aut(M)$ be an algebra automorphism, let $\Lambda$ be a face and let $\sigma_2=\sigma^\Lambda$ be the action on $G(M,(d))$. We say that $\sigma$ induces a face action on $M^{(d)}$ corresponding to $\Lambda$ if there is an automorphism $\sigma_3$ on $M^{(d)}$ such that $(F^\Lambda)^{\sigma_3}=({F^{\sigma_2}})^\Lambda$ for every system of functions in $M$. If such a $\sigma_3$ exists then by abusing the notation we denote it by $\sigma^\Lambda$ as well. Note that if an induced action exists then it has to be unique since linear combinations of elements in $L(M,(d))$ approximate every measurable function in $M^{(d)}$.
\end{definition}

The next lemma yields a characterization of automorphisms that induce face actions on $M^{(d)}$.

\begin{lemma}\label{indactcond} Let $\sigma\in\aut(M)$ and $\Lambda$ be a face in $(d)$. Then $\sigma_2=\sigma^\Lambda$ induces a face action on $M^{(d)}$ if and only if $\tr((F^{\sigma_2})^\Lambda)=\tr(F^\Lambda)$ for every $F\in G(M,(d))$.
\end{lemma}

\begin{proof} First of all we need to show that the induced map $\sigma_3$ is well defined on $L(M,(d))$.
To see that assume that $F_1^{(d)}=F_2^{(d)}$ for two function systems in $G(M,(d))$. The fact that $\sigma$ preserves the trace implies that it preserves scalar products and so $0=\|(F_1)^{(d)}-(F_2)^{(d)}\|_2^2=\|(F_1^{\sigma_2})^{(d)}-(F_2^{\sigma_2})^{(d)}\|_2^2$.
From here lemma \ref{autextends} completes the proof.
\end{proof}

\begin{lemma}\label{indact10} Assume that $\sigma\in G_{d}(M)$ then for an arbitrary $d$ face $\Lambda$ of $(d+1)$ the action $\sigma^\Lambda$ induces a face action on $M^{(d+1)}$. Furthermore the action $\sigma^{(d)}$ induces a face action on $M^{(d)}$.
\end{lemma}

\begin{proof} Both statements follow directly from lemma \ref{gowinvf} and lemma \ref{indactcond}.
\end{proof}

\begin{definition}\label{diagonalaction} In the special case of lemma \ref{indact10} when $\sigma$ is a shift operator, we will call the induced action on $M^{(d)}$ the {\bf diagonal action} of $\bA$. The diagonal action is a simple shift on the space $\bA^{d+1}$ sending $(x,t_1,t_2,\dots,t_d)$ to $(x+t,t_1,t_2,\dots,t_d)$ for a fixed $t\in\bA$.
\end{definition}

\begin{lemma}\label{cubspace1} Let $h:(d)\rightarrow\hat{\bA}_k$ be a map on the $d$ dimensional cube and let $F\in G(L_\infty(\bA),(d))$ be a function system such that $f_S\in h(S)$ for every $S$ in $(d)$.
Assume that $\|\conv_d(F)\|_2\neq 0$. Then for every $d-k+1$ dimensional face $\Lambda=\Lambda(F,K)$ we have
that $$\prod_{S\in\Lambda}h(S)=1.$$
\end{lemma}

\begin{proof} Assume by contradiction that the face $\Lambda$ doesn't satisfy the condition.
For each set $S\subseteq [d]\setminus F$ let us introduce the function
$$g_S=\prod_{H\subseteq F}f_{H\cup S}(x+\sum_{i\in F} t_i).$$
We denote the the above function system by $G$.

For every fixed values of the variables $\{t_i\}_{i\in F}$ each function $g_S$ is contained in the rank one module $\prod_{H\in \Lambda(F,S)}h(H)$. In particular $g_K$ is contained in $b=\prod_{S\in\Lambda}h(S)$ which is not the trivial module by our assumption. In particular we have that $\|g_K\|_{U_k}=0$.
it follows by \ref{convl2norm} that $\|\conv_{k-1}(G)\|_2=0$ for every fixed values of $\{t_i\}_{i\in F}$. On the other hand the function $\conv_d(F)$ is the same as $\conv_{k-1}(G)$ if the variables $\{t_i\}_{i\in F}$ are fixed. This contradicts the assumption that $\|\conv_d(F)\|_2\neq 0$.
\end{proof}

\begin{lemma}\label{cubspace2} Let $A$ be an arbitrary abelian group and let $B$ denote the abelian group of all functions $h:(d)\rightarrow A$ with the point wise multiplication. Let $B_{d,k}$ denote subgroup formed by the elements $h$ with $\prod_{S\in\Lambda}h(S)=1$ for every $d-k+1$ dimensional face $\Lambda$.
Then $B_{d,k}$ is generated by all the elements $g=g(\Lambda,a)$ where $\Lambda$ is a $k$-dimensional face and $a\in A$ such that
$g(S)=1$ if $S\notin\Lambda$ and $g(S)=a^{\epsilon(S)}$ with $\epsilon(S)=(-1)^{|S|}$ if $S\in\Lambda$.
\end{lemma}

\begin{proof} Let $B'_{d,k}$ be the group generated by the elements $g(\Lambda,a)$ where $\Lambda$ is a $k$-dimensional face. First of all note that the elements $g(\Lambda,a)$ are all contained in $B_{d,k}$ since if a $d-k+1$ dimensional face has a non trivial intersection with a $k$-dimensional face then the intersection is a face of dimension at least $1$. This implies that $B'_{d,k}\subseteq B_{d,k}$.

We prove the other containment by induction on $d$. Assume that it is true for $d-1$.
Let $h$ be an element in $B_{d,k}$.
For an arbitrary element $i\in [d]$ we look at the function $h_i(S)=h(S\cup\{i\})$ defined for sets $S\subseteq [d]\setminus\{i\}$. Obviously $h_i$ is in $B_{d-1,k-1}$ and so by induction $h_i\in B'_{d-1,k-1}$ in the cube corresponding to the set $[d]\setminus\{i\}$. Let us consider the map $e(\Lambda,a)=(\Lambda',a)$ where $\Lambda'$ is the $k$-dimensional face in $(d)$ extending $\Lambda$ with the dimension in the $i$-th coordinate. Using this extension operation and the fact that $h_i\in B'_{d-1,k-1}$ we obtain that the function $h'_i$ defined by $h'_i(S)=h_i(S)$ if $i\notin S$ and $h'_i(S)=h_i(S)^{-1}$ if $h\in S$ is in $B'_{d,k}$. Now it is enough to prove that $P_i(h)=h+h'_i$ is in $B'_{d,k}$. The operator $P_i$ behaves as a projection operator to the $d-1$ dimensional face $\Lambda([d]\setminus\{i\},\emptyset)$. Repeating this for different $i$'s we get a function which has to be the unit element at each vertex. This completes the proof.
\end{proof}

\subsection{The diagonal $\sigma$-algebra}

Let $\mathcal{Y}_k$ denote the coset $\sigma$ algebra corresponding to the subgroup
$\{(\bA,0,0,\dots,0)\}$ in $\bA^{k+1}$. The significance of $\mathcal{Y}_k$ is that $k$ degree convolution $\conv_k(F)$ can be obtained by taking the projection of $F^{(k)}$ to $\mathcal{Y}_k$ and then forgetting the first coordinate $x$.

An important property of the induced action of $G_k(M)$ on $\mathcal{B}^{(k)}$ is that it acts trivially on the $\sigma$-algebra $\mathcal{B}^{(k)}\cap\mathcal{Y}_k$.
This follows from the fact that $G_k(M)$ is convolution preserving and $\conv_k(F)$ is projection of $(F)^{(k)}$ to $\mathcal{Y}_k$. Since $\mathcal{Y}_k$ is a coset $\sigma$ algebra it is perpendicular to every shift invariant Hilbert space so the projection to $\mathcal{Y}_k$ from $\mathcal{B}^{(k)}$ is the same as the projection to $\mathcal{B}^{(k)}\cap \mathcal{Y}_k$.

\begin{lemma} An element $\sigma$ of $\aut(M)$ is in $G_k(M)$ if and only if it induces an action on $M^{(k)}$ under which $\mathcal{Y}_k$ is invariant.
\end{lemma}

\begin{proof} Assume that $\sigma$ induces the trivial action on $\mathcal{Y}_k$. Then since projections commute with $\sigma$ we obtain that convolution is preserved. The other direction is trivial.
\end{proof}

\begin{definition} Let $f\in M$ be a function and let $\Lambda$ be a face of the cube $(d)$. We define the function system $t(\Lambda,f)\in G(M,(d))$ by $\{f_S\}_{S\subseteq[d]}$ with $f_S=f^{c(|S|)}$ if $S\in\Lambda$ and $f_S=1$ if $S\notin\Lambda$.
\end{definition}

\begin{lemma}\label{gowgow} Let $F\in G(L_\infty(\bA),(k))$ be a function system such that one of the functions in it has zero $U_k$ norm. Then $\|F^{(k)}\|_{U_k}=0$.
\end{lemma}

\begin{proof} Using the cubic symmetry we can assume without loss of generality that
$\|f_{[k]}\|_{U_k}=0$. The formula for $\|F^{(k)}\|_{U_k}^{2^k}$ is the integral of a product $P$ with $2^{2k}$ terms. For each functions $f_S$ there are $2^k$ terms of the form
$$P_{S,H}=f_S(x+x^H+\sum_{i\in S}(t_i+t_i^H))^{c(|H|)}$$
where $H\subseteq [k]$,
$$x^H=\sum_{j\in H}x_j~,~t_i^H=\sum_{j\in H}t_{i,j}$$
and $c(n)$ is the conjugation operation applied $n$-times.
In the integration the variables $$\{x,t_1,t_2,\dots,t_k,\{x_i\}_{i=1}^k,\{t_{i,j}\}_{1\leq i,j,\leq q}\}$$ are chosen independently.
Now it is easy to see that $P_{[k],[k]}$ is the only term that contains all the variables $\{t_{i,i}\}_{i=1}^k$. Using Fubini theorem we do the integration first only on these variables.
Now for every fixed values for all the other variables the integral becomes
$$\int_{\{t_{i,i}\}_{i=1}^k}f^{c(k)}_{[k]}(c+\sum_{i=1}^k t_{i,i})\prod_{S\subset[k],S\neq [k]}g_S$$
for some functions $g_S$ that depend only on the variables $\{t_{i,i}\}_{i\in S}$.
Each function $g_S$ is measurable in $\sigma_{[k]}^*$ which is the $\sigma$ algebra on $\bA^k$ generated by all the $k-1$ cylindric $\sigma$ algebras. By \cite{Sz1} we know that $\|f_{[k]}\|_{U_k}=0$ is equivalent with saying that the projection of $f^{c(k)}_{[k]}(c+\sum_{i=1}^k t_{i,i})$ to this $\sigma$ algebra is $0$. This completes the proof.
\end{proof}

\begin{lemma}\label{charchar} Let $F\in G(M,(k))$ be a system of functions such that each term is a $k$-th order character in $M$. Then $f=F^{(k)}$ is a $k$-th order character, and it represents the trivial module if and only if all the terms in $F$ represent the trivial module.
\end{lemma}

\begin{proof} It is clear that $\Delta_t f=(\Delta_t F)^{(k)}\in\mathcal{F}_{k-1}(\bA^{k+1})$ for every $t\in\bA^{k+1}$ and so $f$ is contained in a rank one module.
Assume that there is a component $f_S$ in $F$ for some $S\subseteq[d]$ with $f_S\notin\mathcal{F}_{k-1}(\bA)$. Then since $\|f_S\|_{U_k}=0$ lemma \ref{gowgow} completes the proof.
\end{proof}

\begin{lemma}\label{odavissza} Let $\mathcal{B}\subseteq\mathcal{F}_k(\bA)$ be a $\sigma$-algebra.
Then $$(\mathcal{B}\wedge\mathcal{F}_{k-1}(\bA))^{(k)}=\mathcal{B}^{(k)}\wedge\mathcal{F}_{k-1}(\bA^{d+1}).$$
\end{lemma}

\begin{proof} The left side is obviously contained in the right side. To see the other direction assume that $f$ is measurable in the right hand side. The function $f$ can be approximated by finite sums of elements from $L(M,(k))$. It is enough to prove that the projections of the elements of $L(M,(k))$ to $\mathcal{F}_{k-1}$ are measurable in the left hand side. Let $g$ be an element in $L(M,(k))$ of the form $F^{(k)}$ where $F=\{f_S|S\subseteq [k]\}$. For each function $f_S$ we consider the unique decomposition $f_S=f'_S+h_S$ where $f'_S$ is the projection of $f_S$ to $\mathcal{F}_{k-1}$ and $\|h_S\|_{U_k}=0$. Then by lemma \ref{gowgow} we obtain that the $\mathcal{F}_{k-1}$ projection of $F^{(k)}$ is $F'^{(k)}$. The fact that $F'^{(k)}$ is measurable in the left hand side finishes the proof.
\end{proof}

\begin{lemma}\label{invdec} For every $k$-th order character $\phi$ in $M$ there is a $k$-th order character $\phi'$ in $M^{(k)}$ such that $\phi'$ is measurable in $\mathcal{B}^{(k)}\wedge\mathcal{Y}_k$ and $t((k),\phi)/\phi'$ is measurable in $\mathcal{F}_{k-1}(\mathcal{B}^{(k)})$.
\end{lemma}

\begin{proof} Lemma \ref{charchar} implies that $t((k),\phi)$ is a $k$-th order character. In the paper \cite{Sz2} it was proved that $f=\mathbb{E}(t((k),\phi)|\mathcal{Y}_k)$ is not the $0$ function. Since $\mathcal{Y}_k$ is a coset $\sigma$ algebra it is perpendicular to any shift invariant Hilbert space and so $f$ is contained in the same rank $1$ module as $t((k),\phi)$. Let $\phi'$ be a $k$-th order character inside $\mathcal{Y}_k\cap\mathcal{B}^{(k)}$ equivalent to $f$. Then we have that $t((k),\phi)/\phi'$ is measurable in $\mathcal{F}_{k-1}(\mathcal{B}^{(k)})$.
\end{proof}

\begin{lemma}\label{proddec} Assume that $d\geq k$ and let $F\in G(M,(d))$ be a system consisting of functions in $k$-th order rank one modules. Then if $\|\conv_d(F)\|_2$ is not $0$ then $F$ can be written as a product (using the semigroup structure of $G(M,(d))$ of the form
$$F'\prod_{\dim(\Lambda)=k}t(\Lambda,\phi_\Lambda)$$
where $F'$ is some system of functions in $G(M_{k-1},(d))$ and $\phi_\Lambda$ is a $k$-th order character for each $k$-dimensional face $\Lambda$.
\end{lemma}

\begin{proof} The proof is a consequence of Lemma \ref{cubspace1} and \ref{cubspace2}.
\end{proof}

\begin{proposition}\label{szigmadecomp} If $\mathcal{B}\subseteq\mathcal{F}_k(\bA)$ and $d\geq k$ then
\begin{equation}\label{ketszigma}
\mathcal{Y}_d\bigwedge\mathcal{B}^{(d)}=\bigvee_{\dim(\Lambda)=k}(\mathcal{Y}_d\bigwedge\mathcal{B}^\Lambda).
\end{equation}
\end{proposition}

\begin{proof} For every fixed $d$ we prove the statement by induction on $k$. If $k=0$ then it is trivial. Now assume that it is true for $k-1$. In the $\sigma$ algebra $\mathcal{B}^{(k)}$ every function can be approximated by finite sums of expressions $F^{(d)}$ where $F$ is a function system in $G(M,(d))$ consisting of functions contained in rank one modules. Let us denote the set of such function systems by $G'(M,(d))$. We have that in $\mathcal{Y}_d\bigwedge\mathcal{B}^{(d)}$ every function can be approximated as a finite linear combination of functions $\mathbb{E}(F^{(d)}|\mathcal{Y}_d)$ where $F\in G'(M,(d))$. On the other hand $\|\conv_d(F)\|_2=\|\mathbb{E}(F^{(d)}|\mathcal{Y}_d)\|_2$ so it is enough to use elements from $G'(M,(d))$ where $\|\conv_d(F)\|_2$ is not $0$.
Now it suffices to prove that for any such element $F$ the projection $\mathbb{E}(F^{(d)}|\mathcal{Y}_d)$ is measurable in the right hand side of (\ref{ketszigma}).
Let us use lemma \ref{proddec} to rewrite $F$ in the form as it is stated there.
Using lemma \ref{invdec} we rewrite the terms $t(\Lambda,\phi_\Lambda)$ in the form $v_\Lambda w_\Lambda$ such that $v_\Lambda$ is measurable in $\mathcal{F}_{k-1}(\mathcal{B}^\Lambda)$ and $w_\Lambda$ is measurable in $\mathcal{Y}_d\cap\mathcal{B}^\Lambda$.
By lemma \ref{odavissza} the function $v_\Lambda$ is also measurable in $(\mathcal{B}\cap\mathcal{F}_{k-1}(\bA))^\Lambda$.
Putting this together we obtain that
$F^{(d)}$ is a product of the form $AB$ such that $A$ is measurable in the right hand side of ($\ref{ketszigma}$) and $B$ is measurable in $(\mathcal{B}\cap\mathcal{F}_{k-1}(\bA))^{(d)}$.
We get $$\mathbb{E}(AB|\mathcal{Y}_d\wedge\mathcal{B}^{(d)})=A\mathbb{E}(B|\mathcal{Y}_d\wedge(\mathcal{B}\cap\mathcal{F}_{k-1}(\bA))^{d}).$$
Using our induction the proof is complete.
\end{proof}

\subsection{Nilpotency}

\begin{proposition}[Nilpotency]\label{nilp} The group $G_k(M)$ is a $k$-step nilpotent group.
\end{proposition}

before proving this theorem we need some preparation.

\begin{lemma}\label{ukplus} If $\sigma\in G_k(M)$ and $f\in M$ then $\|f\|_{U_{k+1}}=\|f^\sigma\|_{U_{k+1}}$.
\end{lemma}

\begin{proof} In two iterations Lemma \ref{gowinvf} implies that $\sigma^{[k+1]}$ preserves $(F)_{k+1}$.
\end{proof}

\begin{lemma}\label{genface} If $\sigma\in G_k(M)^{(i)}$ where $G_k(M)^{(i)}$ is the $i$-th term in the lower central series (with $G_k(M)^{(0)}=G_k(M)$) then for a system of $2^{k+1}$ functions $F$ in $M$ we have that
$(F^{\sigma^{K}})=(F)$ whenever $K$ is a face in $[k+1]$ of dimension $k-i$.
\end{lemma}

\begin{proof} The proof follows by induction on $i$. Lemma \ref{gowinvf} tells it for $i=0$.
Assume that we have the statement for $i-1$. If $\dim(K)=k-i$ then $K=K_1\cap K_2$ where $\dim(K_1)=k$ and $\dim(K_2)=k-i+1$. If $\sigma_1\in G_k(M)$ and $\sigma_2\in G_k(M)^{(i-1)}$ then $[\sigma_1^{K_1},\sigma^{K_2}]=[\sigma_1,\sigma_2]^K$. This completes the proof.
\end{proof}

now we are ready to prove proposition \ref{nilp}

\begin{proof} Let $\sigma$ be an element in $G_k(M)^{(k)}$. We show that $\sigma$ is trivial.
Let $f\in M$ be a function. Then $\|f-f^\sigma\|_{U_{k+1}}$ has $2^{k+1}$ terms. By lemma \ref{genface} the $\sigma$ can be omitted in each term and so we get $\|f-f^\sigma\|_{U_{k+1}}=\|f-f\|_{U_{k+1}}=0$. This means that $f=f^\sigma$ since $f,f^\sigma\in\mathcal{F}_k$.
\end{proof}

\begin{lemma}\label{autnormpres} The group $G_k(M)$ preserves the uniformity norms $U_i$ for $0\leq i\leq k$.
\end{lemma}

\begin{proof} It is trivial since $\|f\|_{U_i}^{2^k}$ can be expressed as the average of a $k$-th order convolution of copies of $f,\overline{f}$ and the identity function $1$.
\end{proof}

Let $M_i$ denote the subalgebra of $M$ consisting of functions measurable in $\mathcal{F}_i$.

\begin{lemma}\label{mipres} The group $G_k(M)$ preserves the subalgebras $M_i$ for $0\leq i\leq k$.
\end{lemma}

\begin{proof} Let $f\in M_i$ and assume that $f^\sigma$ is not in $M_i$ for some $\sigma\in G_k(M)$. Then $f^\sigma=f_2+f_3$ where $f_2=\mathbb{E}(f|\mathcal{F}_i)$ is in $M_i$ and $\|f_3\|_{U_{i+1}}=0$ but $f_3$ is not $0$. We also have by lemma \ref{autnormpres} that $f_4=f_3^{\sigma^{-1}}$ has zero $U_{i+1}$ norm.
On the other hand the scalar product $(f,f_4)=0$ since $f_4$ is orthogonal to $M_i$. Now $(f^{\sigma},f_4^{\sigma})=(f_2+f_3,f_3)=\|f_3\|_2^2>0$ is a contradiction since $\sigma$ preserves scalar products.
\end{proof}

\begin{lemma}\label{convin2} For $\sigma\in G_k(M)$ and $f,f_2\in M$ we have
\begin{equation}\label{convin2q}
\int_t\|f^\sigma(x)\overline{f_2(x+t)}\|^{2^k}_{U_k}=\int_t\|f(x)\overline{f_2(x+t)}\|^{2^k}_{U_k}.
\end{equation}
\end{lemma}

\begin{proof} Follows directly from lemma \ref{gowinvf}.
\end{proof}

The results in \cite{Sz1} implies the next proposition.

\begin{proposition} The algebra $M_i$ as a module over $M_{i-1}$ can be decomposed into shift invariant rank one modules. The rank one modules are forming an abelian group that we denote by $\hat{M}_i$.
\end{proposition}

Now we prove a strengthening of lemma \ref{mipres}.

\begin{lemma}\label{charpres} Let $\sigma\in G_k(M)$ then the rank one modules in $\hat{M}_i$ (as sets) are preserved under the action of $\sigma$.
\end{lemma}

\begin{proof} Since $G_k(M)$ preserves the algebras $M_i$ it is obviously enough to show the statement for rank one modules in $\hat{M}_k$. Let $f$ be a function contained in such a rank one module and let $f_2$ be a character contained in another rank one module. First note that
\begin{equation}\label{charpresq}
\int_t\|f(x)\overline{f_2(x+t)}\|^{2^k}_{U_k}=0
\end{equation}
By lemma \ref{convin2} we obtain that $\|f^\sigma(x)\overline{f_2(x+t)}\|_{U_k}=0$ for almost all $t$ and so $\mathbb{E}(f^\sigma(x)\overline{f_2(x+t)}|\mathcal{F}_{k-1})=0$ for almost all $t$. Let us fix one such $t$.
Then $$0=\mathbb{E}(f^\sigma(x)\overline{f_2(x+t)}|\mathcal{F}_{k-1})=\mathbb{E}(f^\sigma(x)\overline{f_2(x)}|\mathcal{F}_{k-1})f_2(x)\overline{f_2(x+t)}.$$
This implies that $f^\sigma$ and $f_2$ are orthogonal.
\end{proof}

\subsection{Degree $d$ functions}

Recall that $\mathcal{C}$ is the complex unit circle.
A function $f:\bA\rightarrow\mathcal{C}$ is called a $d$ degree function if $|f(x)|=1$ for every $x$ and $\Delta_{t_1,t_2,\dots,t_{d+1}}f(x)$ is the constant $1$ function.
We denote the set of $d$ degree functions in $M$ by $P_d(M)$.

\begin{lemma} Let $\sigma\in G_k(M)$ such that it induces the trivial action on $M_{k-1}$. Let $f$ be a $k$-th order character in $M$ such that $f^\sigma=fh$. Then $f^{[\iota(-t),\sigma]}=f\Delta_t(h)$. Furthermore $h$ is a $k-1$ degree function.
\end{lemma}

\begin{proof} Using lemma \ref{charpres} we obtain that $h\in M_{k-1}$. From here the first statement is a trivial calculation. The second statement follows from the fact $G_k(M)$ is $k$-step nilpotent and by the application of the first statement iteratively for $-t_1,-t_2,\dots,-t_k$.
\end{proof}

\begin{corollary}\label{alsoszint2} If $\sigma\in G_k(M)$ induces the trivial action on $M_{k-1}$ then there is a homomorphism $\tau:\hat{M}_k\rightarrow P_{k-1}(M)$ such that $f^\sigma=f\tau(a)$ for every $a\in\hat{M}_k$ and $f\in a$.
\end{corollary}

\begin{lemma}\label{purepres} If $f\in P_d(M)$ for some $d\leq k$ and $\sigma\in G_k(M)$ then $f^\sigma/f\in P_{d-1}(M)$.
\end{lemma}

\begin{proof} Note first of all that since $\sigma$ preserves the value distribution we have that $|f^\sigma(x)|=1$ for almost every $x$. Using lemma \ref{convin2} for $f=f_2$ we obtain that
$$\int_t\|f^\sigma(x)\overline{f(x+t)}\|^{2^d}_{U_k}=\int_t\|f(x)\overline{f(x+t)}\|^{2^k}_{U_d}=\|f\|_{U_{k+1}}=1$$
This implies that $\|f^\sigma(x)\overline{f(x+t)}\|^{2^k}_{U_d}=1$ for almost every $t$ since it is at most $1$. Let $t$ be such a value. Then $f^\sigma(x)\overline{f(x+t)}$ has to be in $P_{d-1}(M)$. On the other hand $\overline{f(x)}f(x+t)$ is in $P_{d-1}(M)$ as well so by multiplying them the proof is complete.
\end{proof}

\medskip

The set $P_k(M)$ is an abelian group with respect to multiplication. Lemma \ref{purepres} implies that $G_k(M)$ has a representation of the form $$G_k(M)\rightarrow\aut(P_k(M))$$
which preserves the subgroups $P_d(M)$ and induces the trivial action on the factors $P_d(M)/P_{d-1}(M)$. Let $GP_k(M)$ denote the set of all automorphisms of $P_k(M)$ that preserve the subgroups $P_d(M)$ and act trivially on $P_d(M)/P_{d-1}(M)$. The full group $GP_k(M)$ is also $k$-step nilpotent.

\subsection{The group $G(M)$}

\begin{definition} For an arbitrary function algebra $M=L_\infty(\mathcal{B})$ where $\mathcal{B}$ is a shift invariant sub-algebra of $\mathcal{A}$ we introduce the group $G(M)$ as $\cap_{k=1}^\infty G_k(M)$. In other words $G(M)$ is the group of automorphisms which preserve convolutions of all orders.
\end{definition}

\begin{definition}[Topology] The group $G(M)$ is a topological space with respect to the weakest topology in which the functions $\|f^\sigma-f\|_2$ are all continuous for every $f\in M$.
\end{definition}

The next theorem is crucial.

\begin{theorem}If $\mathcal{B}$ is a shift invariant $\sigma$-algebra in $\mathcal{F}_{k-1}$ and $M$ is the algebra of bounded $\mathcal{B}$ measurable functions (up to $0$ measure change) then $G_k(M)=G(M)$.
\end{theorem}

\begin{proof} We prove by induction that $G_d(M)=G_k(M)$ if $d\geq k$. There is nothing to prove for $d=k$. Assume that $G_{d-1}(M)=G_k(M)$. Lemma \ref{indact10} says that every $\sigma\in G_{d-1}$ induces an action on $\mathcal{B}^{(d)}$. We have to prove that $\mathcal{Y}_d\wedge\mathcal{B}^{(d)}$ is preserved under this action. Using that $\sigma\in G_k(M)$ we have that $\sigma$ preserves all the algebras $\mathcal{B}^\Lambda\wedge\mathcal{Y}_d$ where $\Lambda$ is a $k$-dimensional face. By Proposition \ref{szigmadecomp} the proof is complete.
\end{proof}

In the rest of this chapter we will study extension properties of convolution preserving automorphisms. We will need the next lemma which follows directly from corollary \ref{addconv}.

\begin{lemma}\label{ext1} Let $\sigma$ be an automorphism of $M$ that preserves the $k$-th order rank one modules (elements of $\hat{M}_k$) as sets. Furthermore assume that $\sigma$ preserves the convolution of each system $(F)=\{f_S|S\subseteq [k],f_S\in a^{\epsilon(|S|)}\}$ where $\epsilon(n)=-1^n$ and $a\in\hat{M}_k$ is a fixed rank one module. Then $\sigma\in G_k(M)$.
\end{lemma}

\begin{proposition}\label{ext2} Let $\sigma$ be an automorphism of the algebra $M$ such that it leaves the $k$-th order rank one modules set wise invariant and $\sigma$ restricted to $M_{k-1}$ is in $G(M_{k-1})$. Assume furthermore that for every $k$-th order character $\phi$ the function $f:=t((k),\phi)/\phi'$ (using the notation in lemma \ref{invdec}) satisfies
\begin{equation}\label{kiterj}
\Delta_\sigma(f)=t((k),\Delta_\sigma(\phi)).
\end{equation}
Then $\sigma\in G(M)$.
\end{proposition}

\begin{remark} Note that by lemma \ref{odavissza} $\mathcal{F}_{k-1}(\mathcal{B}^{(k)})=M_{k-1}^{(k)}$ and so $\sigma$ extends to $\mathcal{F}_{k-1}(\mathcal{B}^{(k)})$. This justifies the notation $\Delta_\sigma(f)$.
\end{remark}

\begin{proof} First we prove that $\sigma$ induces an action on $M^{(k)}$.
To see this let $G'((k),M)$ denote the set of function systems in $G((k),M)$ in which every element is from a $k$-th order rank one module. The set of functions $J=\{F^{(k)}|F\in G'((k),M)\}$ is obviously a semigroup satisfying the conditions of lemma \ref{autextends} in $M^{(k)}$ so it is enough to check that $\sigma$ preserves the trace on $J$. Now if one of the modules in $F$ is non trivial then by lemma \ref{charchar} both $\tr(F^{(k)})$ and $\tr((F^\sigma)^{(k)})$ are $0$ since $\sigma$ preserves the rank one modules.
If all the rank one modules are trivial then using $\sigma\in G(M_{k-1})$ shows that the trace is preserved.

Now we prove that $\sigma\in G(M)$. Let $a$ be a fixed module in $\hat{M}_k$ and let $F$ be a system as in lemma \ref{ext1}.  By lemma \ref{ext1} it is enough to check that $\sigma$ is convolution preserving for $F$.
Let $\phi:\bA\rightarrow\mathcal{C}$ be a character representing the module $a$. We write $f_S$ as $\phi^{c(|S|)}g_S$ where $g_S\in M_{k-1}$. Now we have that $F^{(k)}=t((k),\phi)G^{(k)}$. By equation (\ref{kiterj}) $\sigma$ stabilizes the function $t((k),\phi)/f=\phi'$ which is measurable in $\mathcal{Y}_k$. Now using that $\sigma\in G_k(M_{k-1})$
$$\mathbb{E}((F^\sigma)^{(k)}|\mathcal{Y}_k)=\mathbb{E}(t((k),\phi)^\sigma (G^\sigma)^{(k)}|\mathcal{Y}_k)=\mathbb{E}(\phi'f^\sigma(G^\sigma)^{(k)}|\mathcal{Y}_k)=$$
$$=\phi'\mathbb{E}(f^\sigma (G^\sigma)^{(k)}|\mathcal{Y}_k)=\phi'\mathbb{E}(fG^{(k)}|\mathcal{Y}_k)=\mathbb{E}(t((k),\phi)G^{(k)}|\mathcal{Y}_k)=\mathbb{E}(F^{(k)}|\mathcal{Y}_k).$$
\end{proof}

\medskip

\begin{lemma}\label{alsoszint} Let $\tau:\hat{M}_k\rightarrow P_{k-1}(M)$ be an arbitrary homomorphism. Let us define $\sigma$ as the unique automorphism satisfying that $\sigma$ restricted to $M_{k-1}$ is trivial and $f^\sigma=f\tau(a)$ for any function $f$ in a rank one module $a\in\hat{M}_k$.
Then $\sigma\in G(M)$.
\end{lemma}

\begin{proof} The fact that $\sigma$ extends to an automorphism follows for lemma \ref{autextends} applied to the semigroup of functions contained in rank one modules. (It is obvious that $\sigma$ preserves the trace operation since $\tr(f)=0$ for any function in a non trivial rank one module. On the other hand $f^\sigma=f$ for functions in the trivial rank one module $M_{k-1}$.)

Now it is enough to check the conditions of proposition \ref{ext2}. Since $\Delta_\sigma\phi$ is in $P_{k-1}(M)$ we have that $t((k),\Delta_\sigma\phi)$ is the constant $1$ function on the right hand side of (\ref{kiterj}). The left hand side of (\ref{kiterj}) is also $1$ since $\sigma$ acts trivially on $M_{k-1}$.
\end{proof}

\section{The linear case and the quadratic case}

In this part of the paper we explain our theory in the special case when $\mathcal{B}$ is a shift invariant $\sigma$-algebra embedded into $\mathcal{F}_1$ or $\mathcal{F}_2$.

\subsection{The linear case}

First of all note that if $\mathcal{B}\subseteq\mathcal{F}_1$ is a shift invariant $\sigma$ algebra then it is fully described by the rank one modules in the first dual group $\hat{\bA}_1$ contained in $M=L_\infty(\mathcal{B})$. We have that $\hat{M}_1$ is a subgroup in $\hat{\bA}_1$ characterizing $\mathcal{B}$.

Let $\sigma\in G(M)=G_1(M)$. Since every element $\chi$ in every rank one module is in $P_1(M)$ we have by lemma \ref{purepres} that $\chi^\sigma=\chi\tau(\chi)$ where $\tau(\chi)$ is a constant depending only on the rank one module containing $\chi$. Note that in the linear case, elements of rank one modules are fully determined up to constant multiples.
In other words $\tau$ is a homomorphism from $\hat{M}_1$ to the unit circle $\mathcal{C}$ in $\mathbb{C}$.

On the other hand by lemma \ref{alsoszint} every such homomorphism corresponds to an element in $G(M)$. It follows that
$$G(M)=\hom(\hat{M}_1,\mathcal{C})$$
and in particular
$$G(L_\infty(\mathcal{F}_1))=\hom(\hat{\bA}_1,\mathcal{C}).$$

An interesting fact is that is that every element of $G(M)$ extends to an element of $G(L_\infty(\mathcal{F}_1))$ since $\mathcal{C}$ is an injective $\mathbb{Z}$ module (it is divisible) and so every homomorphism from $\hat{M}_1$ to $\mathcal{C}$ extends to a homomorphism from $\hat{\bA}_1$ to $\mathcal{C}$.

It follows from Tchihonov's theorem that $G(M)$ is a compact topological group. The group of translations $\bA$ is embedded into $G(M)$ as a dense subgroup. This density is explained by the next lemma.

\begin{lemma}\label{linrest} Let $\sigma$ be an element in $G(L_\infty(\mathcal{F}_1))$. For every separable shift invariant subalgebra $\mathcal{B}$ of $\mathcal{F}_1$ there is an element $t\in\bA$ such that the restriction of $\sigma$ to $\mathcal{B}$ is the same as the shift operator $\iota(t)$ on $\mathcal{B}$.
\end{lemma}

\begin{proof} The fact that $\mathcal{B}$ is separable is equivalent with the fact that it is generated by a countable group $H\subset\hat{\bA}_1$ of linear characters. Let $\hat{H}$ denote the dual group $\hom(H,\mathcal{C})$ of the discrete group $H$. Now $\hat{H}$ is a separable compact topological group. The group $\bA$ has a homomorphism $h$ into $\hat{H}$ given by $h(t)(\chi)=\chi(t)$. Since linear characters are continuous in the $\sigma$ topology on $\bA$ we have that the image of $\bA$ in $\hat{H}$ is a compact subgroup $\hat{H}'$. Using the fact that the linear characters are independent on $\bA$ the dual group of $\hat{H}'$ is $H$ which means that $\hat{H}'=\hat{H}$
\end{proof}

\subsection{Harmonic analytic limits of functions on abelian groups}

Let $f_1,f_2,\dots$ be a uniformly bounded sequence of functions defined on a growing sequence of finite abelian groups $\{A_i\}_{i=1}^\infty$ with ultra product $\bA$. We denote the ultra limit function $\limo f_i$ by $f$. The function $f$ is in $L_\infty(\bA)$.
Let $f'$ be the projection $\mathbb{E}(f|\mathcal{F}_1)$ and let $f'=\lambda_1\phi_1+\lambda_2\phi_2+\dots$ be its first order Fourier decomposition where $\phi_i$ are distinct measurable linear characters of $\bA$ and $\lambda_i$ are nonzero complex numbers.
Let $Q$ denote the discrete countable abelian subgroup in $\hat{\bA}_1$ generated by $\phi_1,\phi_2,\dots$. The Pontrjagin dual of $G$ is a compact topological group $G$.

\begin{lemma} The group $G$ is isomorphic to the image of $\bA$ under the map
$$\tau:a\mapsto \{\phi(a)\}_{\phi\in Q}\subseteq \mathcal{C}^Q.$$
\end{lemma}

\begin{proof} Clearly, for every fix $a\in\bA$ we have that $\phi\mapsto\phi(a)$ defines an element of the dual group of $Q$. It is enough to prove that every homomorphism $Q\rightarrow\mathcal{C}$ arises this way. First of all we show that the image of $\bA$ under the map $\tau$ is a compact subgroup in $\mathcal{C}^Q$. Since every linear character of $\bA$ is continuous in the $\sigma$ topology (they are ultra limits of ordinary characters) the map $\tau$ is also continuous in the $\sigma$-topology. This means that the image of $G$ under $\tau$ is countably compact. Since it is in $\mathcal{C}^Q$, it has to be compact.
Obviously, the dual group of $\tau(\bA)$ is isomorphic to $Q$ and so it is isomorphic to $G$.
\end{proof}

\medskip

Since $f$ is measurable in the $\sigma$ algebra generated by $\tau$ there is a measurable function  $h:G\rightarrow\mathbb{C}$ such that $f(a)=h(\tau(a))$ for almost every $a$.
The function $h$ can be obtained by the formula $\lambda_1\phi^*_1+\lambda_2\phi^*_2+\dots$ where $\phi^*_1,\phi^*_2,\dots$ are the characters of $G$ corresponding to $\phi_1,\phi_2,\dots$.

\begin{definition} The measurable function $h$ on the compact topological group $G$ is called the harmonic analytic limit of the sequence $\{f_i\}_{i=1}^\infty$.
\end{definition}

\begin{remark} The above definition also makes sense if $\{A_i\}_{i=1}^\infty$ is itself a sequence of compact abalian groups.
\end{remark}

The previous remark gives rise to interesting examples. It can happen for example that $\{f_i\}$ is a sequence of measurable functions of the group $\mathbb{R}/\mathbb{Z}$ but the harmonic analytic limit exists on $(\mathbb{R}/\mathbb{Z})^2$.

\subsection{The quadratic case}

We start by analyzing simple $\sigma$-algebras in $\mathcal{F}_2$. Let $\mathcal{B}=\mathcal{B}_\phi$ be the $\sigma$ algebra generated by $\mathcal{F}_1$ and a single quadratic character $\phi$.
Since $\Delta_t \phi^n$ is in $\mathcal{F}_1$ for every $t\in\bA$ and integer $n$ we have that $\mathcal{B}_\phi$ is shift invariant.
Furthermore $\hat{M}_2$ is a cyclic subgroup of $\hat{\bA}_2$ since it is generated by the module of $\phi$.

Let $Q$ be the subgroup in $G(M)$ generated by those elements whose restriction to $M_1$ is trivial. Obviously $Q$ is normal in $G(M)$. We know by lemma \ref{alsoszint} and corollary \ref{alsoszint2} that $Q$ can be described as the group $\hom(\hat{M}_2,P_1(M))$.

Note that $P_1(M)$ consists of liner characters of $\bA$ multiplied by root of unities and so $$P_1(M)=\hat{\bA}_1\times\mathcal{C}.$$
By abusing the notation Let $\bA$ be the group of shifts on $\bA$ which is also embedded into $G(M)$.
The conjugation action of $\bA$ is given by
$(\chi,c)^t=(\chi,c\chi(t))$. Now the semidirect product
$$\bA\ltimes(\hat{\bA}_1\times\mathcal{C})$$
is embedded into $G(M)$. As the next lemma says the group $G(M)$ is bigger then that.
First we examine the case when the module of $\phi$ is of infinite order but as we will see after the lemma this assumption is not crucial. Note that in the so called ``zero characteristic case'' \cite{Sz2} the group $\hat{\bA}_2$ is torsion free.

\begin{lemma}\label{ext15} Assume that the module of $\phi$ has infinite order.
Then every element $\sigma\in G(M_1)$ extends to some element in $G(M)$.
\end{lemma}

\begin{proof} Let $\phi'$ be a character guaranteed by lemma \ref{invdec}. We know that $t((k),\phi)/\phi'$ is in $M_1^{(2)}$ and so there is a separable (and shift invariant) sub $\sigma$-algebra $\mathcal{B}'$ of $\mathcal{F}_1$ such that $t((k),\phi)/\phi'$ is measurable in $(\mathcal{B}')^{(2)}$. By lemma \ref{linrest} there is an element $t\in\bA$ such that $\sigma$ restricted to $\mathcal{B}'$ is the shift $\iota(t)$.

Now we define $\phi^\sigma$ as $\phi^{\iota(t)}=\phi\Delta_t\phi$. This determines the action of $\sigma$ uniquely on $\mathcal{B}_\phi$ because everything in $M$ can be approximated by a ``polynomial'' of the form $a_0+a_1\phi+a_2\phi^2+\dots+a_n\phi^n$ where $a_i\in M_1$. The condition that $\phi$ has infinite order guarantees that $\sigma$ is well defined and multiplicative on such polynomials. It is also clear that $\sigma$ preserves the trace (and conjugation) since the trace in any non trivial module is $0$. We use lemma \ref{autextends} to show that $\sigma$ determines an automorphism on $M_2$.

Now to prove that $\sigma$ is in $G(M)$ we need to check the conditions of proposition \ref{ext2}.
The only non-trivial condition is the equation (\ref{kiterj}).
Using the fact that $\sigma$ acts as $\iota(t)$ on $\mathcal{B}'$ we get that
$\Delta_\sigma f=\Delta_t f$ where we use the diagonal shift operator on $M^{(2)}$.
On the other hand $t((k),\Delta_\sigma\phi)=t((k),\Delta_t\phi)$.
Now since $t((k),\phi)/f$ is invariant under the diagonal shifts we get that
$\Delta_t f=t((k),\Delta_t\phi)$. This completes the proof.
\end{proof}

The only place where we needed that $\phi$ is of infinite order is to guarantee that $\sigma$ is well defined. However if we extend the $\sigma$-algebra $\mathcal{B}'$ in the proof by the $\sigma$-algebra generated by $\phi^{o(\phi)}$ (where $o(\phi)$ is the order of the module of $\phi$) then the reader can check easily that $\sigma$ becomes well defined.

Summarizing of our results we get the short exact sequence

$$0\rightarrow \hat{\bA}_1\times\mathcal{C}\rightarrow G(M)\rightarrow \hom(\hat{\bA}_1,\mathcal{C})\rightarrow 0$$

which describes the structure of $G(M)$. In other words $G(M)$ is an extension of $Q=\hat{\bA}_1\times\mathcal{C}$ by $G(M_1)=\hom(\hat{\bA}_1,\mathcal{C})$.
The action of $G(M_1)$ on $Q$ does not depend on the choice of $\phi$ however the extension is not always a semi direct product.

Let $Z$ be the group of elements in $Q$ of the form $(1_{\bA},c)$. The group $Z$ is isomorphic to the dual group of the cyclic group generated by the module of $\phi$. It is either the full unit circle or a finite subgroup of it. The group $Z$ is obviously in the center of $G(M)$.

\begin{lemma} The group $Z$ contains the commutator subgroup of $G(M)$.
\end{lemma}

\begin{proof} Any commutator is contained in the center of $G(M)$ since $G(M)$ is two-nilpotent. Furthermore the commutator is contained in $Q$ since $G(M)/Q$ is abelian. This means that the commutator subgroup is contained in $Q\cap Z(G(M))$. Now let $g=(\chi,c)$ be an arbitrary element in $Q$. If $\chi$ is non trivial then there is an element $t\in \bA$ such that $\chi(t)\neq 1$. Then the shift operator $t$ does not commute with $g$ since $\phi^{g\iota(t)}\neq\phi^{\iota{t}g}$.
\end{proof}

Now that we have a full description of $G(M)$ we get the following lemma.

\begin{lemma}\label{deltasigma} If $T$ is the type of $\phi$ then the dual support of $\Delta_\sigma\phi$ is contained in some coset of $T$.
\end{lemma}

\begin{proof} Let $\sigma_2$ be the image of $\sigma$ under the map $G(M)\rightarrow G(M_1)$. Then $\sigma$ is one of the extensions of $\sigma_2$ to $G(M)$. By the proof of lemma \ref{ext15} we have know that $\Delta_\sigma\phi=(\Delta_t\phi)\chi$ for some $t\in\bA$ and $\chi\in Q$. This completes the proof.
\end{proof}

\subsection{The stabilizer of $\phi$}

In this part we analyze the stabilizer of the function $\phi$ in $G(M)$.
Let us start by analyzing the stabilizer $\stab(G(M),\phi)$ of the function $\phi$ in $G(M)$.
Recall \cite{Sz2} that there is a separable shift invariant $\sigma$-algebra $\mathcal{B}_2$ in $\mathcal{F}_1$ in which all the functions $\Delta_{t_1,t_2}\phi$ are measurable.
This $\sigma$-algebra is spanned by a countable group $T\subseteq \hat{\bA}_1$. We will denote by $\ker(T)$ the subgroup of $G(M_1)=\hom(\bA_1,\mathcal{C})$ consisting of those homomorphisms that vanish on $T$. The subgroup $T$ is called the type of $\phi$.

\begin{lemma}\label{kertesz} Let $\sigma$ be an element in $\ker(T)$ and $\sigma_2$ be an extension guaranteed by lemma \ref{ext15}. Then $\phi^{\sigma_2}=\phi\chi c$ where $\chi$ is some linear character and $c$ is a constant of absolute value $1$.
\end{lemma}

\begin{proof} We learn from the proof of lemma \ref{ext15} that the action of every extension $\sigma_2$ of $\sigma$ on $\phi$ is of the form $\phi^{\sigma_2}=\phi\Delta_t\phi\chi_1 c_1$ for an arbitrary element $t$ which acts the same way on $\mathcal{B}'$ as $\sigma$.
Let us consider the $\sigma$-algebra generated by $\mathcal{B}'$ and $\mathcal{B}_2$. This is still a separable $\sigma$ algebra so we can find a $t$ which acts on this in the same way as $\sigma$.
It is enough to prove that $\Delta_t\phi$ is also of the form $\chi_1c_1$. In other words $\Delta_{t_1,t_2,t}\phi=1$ for every $t_1,t_2$.
This follows from $\Delta_{t_1,t_2,t}\phi=\Delta_t(\Delta_{t_1,t_2}\phi)$ because $\Delta_{t_1,t_2}\phi$ is measurable in $\mathcal{B}_2$ and so the shift by $t$ is acts trivially on it.
\end{proof}

\begin{corollary} Let $\sigma$ be an element in $\ker(T)$ then $\sigma$ has an extension $\sigma_2$ to $G(M)$ such that $\sigma_2$ stabilizes $\phi$.
\end{corollary}

\begin{proof} Let us first consider an arbitrary extension and then multiply it by the element in $Q$ which takes $\phi$ to $\phi\chi^{-1}c^{-1}$.
\end{proof}

\begin{corollary}\label{stabchar} The image of $\stab(G(M),\phi)$ under the homomorphism $G(M)\rightarrow G(M_1)$ is exactly $\ker(T)$ and furthermore $\stab(G(M),\phi)\cap Q=1$.
\end{corollary}

\begin{proof} The second statement is trivial sice only one extension of each $\sigma\in G_1(M)$ can be good. To see the first statement assume that $\sigma$ is not in $\ker(T)$. Then there is an element $a$ in $T$ such that $\sigma(a)\neq 0$ and the $a$ component of $\Delta_{t_1,t_2}\phi$ is not zero for some $t_1,t_2$. This implies that $\Delta_t(\Delta_{t_1,t_2}\phi)$ can't be the constant $1$ function for any $t$ acting in the same way on $\mathcal{B}_2$ as $\sigma$.
\end{proof}

Our goal in this chapter is to extract a group from $G(M)$ which has a ``standard'' (the opposite of non-standard) structure (it will be a locally compact Hausdorff topological group).
We denote by $W$ the centralizer of $\stab(G(M),\phi)$ is $G(M)$.
It will out that $W_0=W/\stab(G(M),\phi)$ which is a section of $G(M)$ is a standard topological group which is crucial in the study of our situation.

Recall that $T$ is the type of $\phi$.

\begin{lemma}\label{wchar} The group $W$ consists of those elements $\sigma\in G(M)$ for which $\Delta_\sigma \phi$ is measurable in the $\sigma$-algebra $\mathcal{B}_2$ generated by the rank one modules in $T$.
\end{lemma}

\begin{proof} First we prove that if $\sigma$ centralizes $W$ then $\Delta_\sigma\phi$ is measurable in $\mathcal{B}_2$. Assume by contradiction that $\Delta_\sigma\phi$ is not measurable in $\mathcal{B}_2$. Then by corollary \ref{stabchar} there is a $\sigma_2$ in $\stab(G(M),\phi)$ such that $(\Delta_\sigma\phi)^\sigma_2\neq\Delta_\sigma \phi$. Now $\phi^{\sigma\sigma_2}=\phi(\Delta_\sigma\phi)^{\sigma_2}$ is not the same as $\phi^{\sigma_2\sigma}=\phi\Delta_\sigma\phi$ and so $\sigma$ can't be in the centralizer.

A concrete calculation shows the other containment.
If $\Delta_\sigma\phi$ is measurable in $\mathcal{B}_2$ then $(\Delta_\sigma\phi)^{\sigma_2}=\Delta_\sigma\phi$ and so $\phi^{\sigma\sigma_2}=\phi^{\sigma_2\sigma}$.
\end{proof}

\medskip

Let $T^*=\{(\chi,c)|\chi\in T,c\in\mathcal{C}\}\subset Q$.

\begin{corollary} $W\cap Q=T^*$
\end{corollary}

\begin{corollary} Every $\sigma\in G(M_1)$ has an extension $\sigma_2$ to an element in $G(M)$ such that $\sigma_2\in W$.
\end{corollary}

\begin{proof} Let $\sigma_3$ be an arbitrary extension. We know that the dual support of $\Delta_{\sigma_3}\phi$ is a coset of $T$. Then there is a linear character $\chi$ such that the dual support of $\Delta_{\sigma_3}\phi\chi$ is a subset of $T$. This by lemma \ref{wchar} means that $\sigma_3$ multiplied by the element $(\chi,1)$ form $Q$ is an extension in $W$.
\end{proof}

\medskip

Summarizing our knowledge, we have learned that $W$ is a subgroup in $G(M)$ which maps surjectively onto $G(M_1)$ under the map $G(M)\rightarrow G(M_1)$. Furthermore the intersection of $W$ by $Q$ is $T^*$. Let $\tau$ denote the homomorphism  $W\rightarrow W/\stab(G(M),\phi)$. Obviously $\tau$ is an embedding restricted to $T^*$ because $\stab(G(M),\phi)\cap T^*$ is trivial. The group $\tau(T^*)$ is a normal subgroup in $W_0$. The factor group $W_0/\tau(T^*)$ is isomorphic to a compact abelian group $A$ which is the Pontrjagin dual of $T$.
Let us write $T^*$ in the form $T_0\times Z$ where $T_0$ is the group of linear characters representing the modules in $T$. In particular $T$ and $T_0$ are naturally isomorphic.
Since $Z$ is a compact subgroup of the unit circle we have that $\tau(T_0)$ is a co-compact subgroup in $W_0$.

Now we introduce the topology on $W_0$ that is relevant for us.
First of all Let us consider the separable Hilbert space of functions $\mathcal{H}_\phi=\{f\phi|f\in L_2(\hat{T})\}$. In this definition $f$ is a measurable function in $\mathcal{F}_1$ whose character decomposition contains only elements from $T$. The space of such functions is naturally isomorphic with $L_2(\hat{T})$.
Before stating the next lemma note that since $L_\infty(\mathcal{B})$ is dense in $L_2(\mathcal{B})$ and $G(M)$ is scalar product preserving, there is a natural unitary representation of $G(M)$ on $L_2(\mathcal{B})$.

\begin{lemma} The space $\mathcal{H}_\phi$ is invariant under the action of $W$. We denote the representation of $W$ on $\mathcal{H}_\phi$ by $\Xi:W\rightarrow U(\mathcal{H}_\phi)$. The kernel of $\Xi$ is $\stab(G(M),\phi)$ and so $\Xi$ induces a faithful representation $\Xi:W_0\rightarrow U(\mathcal{H}_\phi)$.
\end{lemma}

\begin{proof} Lemma \ref{wchar} shows the first statement. The second statement follows from lemma \ref{stabchar}.
\end{proof}

\begin{definition}[Topology on $W_0$] The topology on $W_0$ is given by the week operator topology on $\Xi(W_0)$.
\end{definition}

\subsection{Quadratic nil-patterns}

To be able to reproduce the separable structures appearing on $G(M)$ on other abelian groups we introduce structures that we call quadratic nil-pattern.
The use of these structures is that they make the investigation of $W_0$ independent from the big non-standard group $G(M)$ and also help in transporting our results to the finite case.

\begin{definition}[Quadratic nil-pattern] A (quadratic, elementary) nil-pattern $N$ of type $(T,Z)$ is locally compact Hausdorff group which is the element of the following short exact sequence
$$0\rightarrow T\times Z\rightarrow N\rightarrow \hat{T}\rightarrow 0.$$
and has the following properties:
\begin{enumerate}
\item $N,T,\hat{T},Z$ are topological groups and the arrows are continuous. The second arrow is a homeomorphism of $T\times Z$ with its image and the third is homeomorphism between the factor space $N/(T\times Z)$ and $\hat{T}$.
\item $N'\subseteq Z$
\item $Z$ is a compact subgroup of the circle group $\mathcal{C}$
\item $T$ is a countable abelian group with the discrete topology and $\hat{T}$ is the dual of $T$ with the usual compact topology.
\item the action of $\hat{T}$ on $T\times Z$ is given by $(a,b)^\chi=(a,b\chi(a))$
\end{enumerate}
\end{definition}

Note that since $\hat{T}$ and $Z$ are compact, the group $T$ is a co-compact subgroup of $N$. We will call the left coset space of $T$ the {\bf core of the nil-pattern}. We will denote the core by $C(N)$. The Haar measure induces a unique $N$ invariant probability measure on the core. Even though the core is not a group there is a surjective map $c:C(N)\rightarrow\hat{T}$ induced by the group homomorphism $N\rightarrow\hat{T}$. This is well defined since $T$ is in the kernel of this map. We call the map $c$ the {\bf first degree map}.
It also follows form the definition that $Z$ is in the center of $N$ and $N$ is two step nilpotent. If $T$ is trivial then we say that the {\bf nil-pattern is trivial}, if $T$ is finitely generated then we say that the {\bf nil-pattern is finitely generated}.
Two nil-classes are called isomorphic if the extensions are isomorphic.

\begin{remark} Since the group $Z$ is a compact subgroup of the circle group $\mathcal{C}$ (and most of the times it is isomorphic to $\mathcal{C}$, we usually omit $Z$ from the specification of the type of $N$, and we will just say that $N$ is of type $T$.
\end{remark}

\begin{remark} Nil-pattern of type $T$ can be classified by elements of the second co-homology group $H^2(\hat{T},T\times Z)$. Note that not every element in $H^2(\hat{T},T\times Z)$ correspond to a nil-pattern because of the extra conditions $Z\subseteq N'$ and the continuity of the arrows.
\end{remark}

The motivation for the definition of nil-pattern is that we are able to re-interpret them over other abelian groups and so each nil-pattern can be used to produce a class of nil-systems. To be more precise, for every homomorphism $T\rightarrow T_2$ there is a way of creating a nil-pattern out of $N$. We call these new nil-patterns {\bf interpretations} of $N$.

Let $\alpha:T\rightarrow T_2$ be a homomorphism into an abelian group $T_2$. Then alpha induces a map $\hat{\alpha}:\hat{T_2}\rightarrow\hat{T}$ and this map induces a map $\hat{\beta}$ from $H^2(\hat{T},T\times Z)$ to $H^2(\hat{T_2},T\times Z)$. On the other hand $\alpha$ induces a map $\beta$ from $H^2(\hat{T}_2,T\times Z)$ to $H^2(\hat{T}_2,T_2,\times Z)$. By composing $\hat{\beta}$ and $\beta$ we get a map
$$\alpha^*:H^2(\hat{T},T\times Z)\rightarrow H^2(\hat{T_2},T_2\times Z).$$

\begin{definition}[Interpretation] The image of the co-homology class describing the extension of $N$ under the map $\alpha^*$ will be called the the interpretation of the nil-pattern under the homomorphism $T\rightarrow T_2$. By abusing the notation we denote the interpretation of $N$ under the map $\alpha$ by $\alpha^*(N)$. As we will see later $\alpha$ induces a map from $C(\alpha^*(N))$ to $C(N)$ so its acts in the reverse way between the cores.
\end{definition}

\begin{definition}[Class] We say that a nil-pattern $N_2$ belongs to the class of $N$ if it is an interpretation on $N$ over some homomorphism.
\end{definition}

\subsection{Some group theory}

In this chapter we describe two product notion for groups that we will need later.

\noindent{\bf Sub-direct product over a joint factor:}~
Let $G_1$ and $G_2$ be two groups and two surjective homomorphisms
$\phi_1:G_1\rightarrow H$ and $\phi_2:G_2\rightarrow H$.
Then we define $G_1\times_H G_2$ as the subgroup
$$\{(a,b)|\phi_1(a)=\phi_2(b)~,a\in G_1,b\in G_2\}$$
of the direct product $G_1\times G_2$.

\bigskip

\noindent{\bf Amalgamated semi direct product:}~
Let $K$ be a group and $\phi$ be a homomorphism of $K$ into $\aut(N)$. Let $H$ be a group which has injective homomorphisms of the form $\phi_1:H\rightarrow N$ and $\phi_2:H\rightarrow K$.
Assume furthermore that for every $h\in H$ and $k\in K$ we have $$\phi_2(\phi_1^{-1}(\phi_1(h)^{\phi(k)}))=k^{-1}\phi_2(h)k.$$
Then there is a unique group $K\ltimes_H N$ which contains $K$ as a subgroup, $N$ as a normal subgroup, $G=KN$, $K\cap N$ is $H$ and $K$ acts on $N$ by $\phi$.
The group $G$ can be obtained by first taking the amalgamated product $K*_\phi N$ and then by factoring out the relations
$$\{k^{-1}nk=n^{\phi(k)}|n\in N,k\in K\}.$$

\medskip

\subsection{Interpretations and related maps}

\medskip

In this part we give a more explicit definition of the notion of an ``interpretation'' of a nil-pattern. The proof that the two kind of definitions are the same will be left to the reader. In the rest of the paper we use the definition worked out in this chapter.

Recall that if $N$ is a quadratic nil-pattern of type $T$ and $\alpha:T\rightarrow T_2$ is a homomorphism than it creates an interpretation of $N$ which is a nil-pattern of type $T_2$.
Note that if a map $\alpha:T_2\rightarrow\hat{T}$ is given then it induces a map $\hat{\alpha}:T\rightarrow \hat{T_2}$. In this situation we will say that the interpretation corresponding to $\alpha$ is the interpretation obtained from $\hat{\alpha}$ which is a nil pattern of type $\hat{T_2}$. This will be important later on.

Every homomorphism $\alpha:T\rightarrow T_2$ can be written as a composition of an epimorphism and a monomorphism. For this reason we divide our discussion into these two cases

\medskip

\noindent{\bf Epimorphisms:}~~

Assume that $\alpha:T\rightarrow T_2$ is a surjective map. This means that $T_2$ is a factor group of $T$ by the kernel $T_3$ of $\alpha$.

\begin{definition}[interpretation (factor)] Let $N$ be a nil-pattern of type $(T,Z)$, let $T_3$ be a subgroup of $T$ and let $N'$ be the pre-image of $\ker(T_3)$ (which is the dual of $T/T_2$) under the homomorphism $N\rightarrow\hat{T}$. Then $T_3$ is in the center of $N'$ and the group $N_2=N'/T_3$ is a nil-pattern of type $(T_2,Z)$ which is the interpretation of $N$ under the map $\alpha:T\rightarrow T_2$.
\end{definition}

Observe that $N_2$ is a factor group of $N'$ in a way that the corresponding normal subgroup is contained in $T$. It follows that $C(N_2)$ can be identified with the left coset space of $T$ in $N'$. This means that $C(N_2)$ is naturally embedded into $C(N)$. In other words an epimorphism $\alpha:T\rightarrow T_2$ induces a monomorphism $C(N_2)\rightarrow C(N)$ where $N_2$ is the interpretation of $N$ over $\alpha$.
We will see that the dual statement is also true when $\alpha$ is a monomorphism.


\medskip

\noindent{\bf Monomorphism:}~~

Assume that $\alpha:T\rightarrow T_2$ is an injective map. In this case we have a surjective map $\hat{\alpha}:\hat{T_2}\rightarrow \hat{T}$.
This means that $\hat{T}$ is a factor group of both $N$ and $\hat{T_2}$.
Let $K=N\times_{\hat{T}} \hat{T_2}$ be the sub-direct product over the common factor $\hat{T}$.
Let $M$ denote the group $T_2\times\mathcal{C}$ and $H$ denote the group $T\times\mathcal{C}$.
Let $\phi_1:H\rightarrow M$ and $\phi_2:H\rightarrow K$ be the natural embeddings.
Let $\tau:K\rightarrow\hat{T_2}$ denote the projection to the ``second coordinate'' in $K$. We define the action $\phi:K\rightarrow\aut(M)$, as usual, by
$$k^{-1}(t,c)k=(t,c\tau(k)(t)).$$
Let $N_2=K\ltimes_H M$ be the amalgamated semidirect product. Then $M$ is a normal subgroup in $N_2$. The factor group is isomorphic to $K/N$ which is $\hat{T_2}$. It follows that $N_2$ is a quadratic nil-pattern of type $T_2$.

\begin{definition}[interpretation (embedding)] The amalgamated semidirect product $N_2=K\ltimes_H M$ is called the interpretation of $N$ under the map $\alpha$.
\end{definition}

Observe that the only difference between $K$ and $N_2$ is that $T_2$ is increased to $T$. In other words there is a natural bijection between the left coset space of $T_2$ in $N_2$ and the left coset space of $T$ in $K$. On the other hand $K$ has a homomorphism to $N$ whose kernel has trivial intersection with $T_2$. This means that this homomorphism induces a map between the left coset spaces of $T_2$ in $K$ and the left coset space of $T_2$ in $N$. Summarizing this we get an epimorphism from $C(N_2)$ to $C(N)$. Note that the fibres of this epimorphism are imprimitivity domains of the action of $K$ on $C(N_2)$.


\medskip

\noindent{Lift:}~Let $A$ be a finite group and let $\psi:A\rightarrow C(N)$ be a map such that its composition with the first degree map $c:C(N)\rightarrow\hat{T}$ is a homomorphism $\alpha$.
In this case the dual map $\hat{\alpha}:T\rightarrow\hat{A}$ creates an interpretation $N_2$ of type $\hat{A}$. As we have seen, this interpretation induces a map $g:C(N_2)\rightarrow C(N)$.
\begin{definition} We define the {\bf lift} $\psi_2:A\rightarrow C(N_2)$ of $\psi$ as the unique map satisfying the next two conditions.
\begin{enumerate}
\item $\psi_2\circ g=\psi$
\item $\psi_2$ composed with the first order map of $C(N_2)$ is the identity map on $A$.
\end{enumerate}
\end{definition}

\subsection{Quadratic nil-morphisms}

The general concept of a quadratic nil-morphism was introduced in the introduction.
Here we study nil-morphisms into nil-manifolds coming from nil-patterns.

\begin{definition} We say that a nil-morphism is normalized if the unit element of $A$ is mapped into the coset of the unit element in $N$.
\end{definition}

Let us fix an elementary quadratic nil-pattern $0\rightarrow T\times Z\rightarrow N\rightarrow \hat{T}\rightarrow 0$. Let $A$ be a finite abelian group and let $\psi:A\rightarrow C(N)$ be a normalized nil-morphism. Let $c:C(N)\rightarrow \hat{T}$ denote the first degree map. It is obvious from the definitions that the composition $\alpha=\psi\circ c$ is a homomorphism from $A$ to $\hat{T}$. The map $\hat{\alpha}:T\rightarrow\hat{A}$ gives an interpretation $N_2$ of $N$ of type $\hat{A}$. Let $\psi_2$ denote the lift of $\psi$ to $C(N_2)$.

\begin{lemma}\label{liftnilhom} The map $\psi_2$ is a normalized nil-morphism.
\end{lemma}

\begin{proof} Let $\hat{T_2}\subseteq\hat{T}$ denote the image of $\alpha$ and let $N'$ be the primage of $\hat{T_2}$ under the homomorphism $N\rightarrow \hat{T}$. Let $T_2$ denote the dual of $\hat{T_2}$. With this notation the range of $\psi$ consists of cosets of $T$ contained in $N'$. As we have seen, the group $N_3=N'/T_2$ is a nil pattern of type $T_2$. It is easy to see from the definition that $\psi$ restricted to $C(N_3)$ is also a quadratic nil character. This means that without loss of generality we can assume that $\alpha$ is surjective.

New let $K$ denote the sub-direct product of $N$ with $A$ over the common factor $\hat{T}$.
The space $C(N_2)$ is the same as the left coset space of $T$ in $K$.
Let $g'_a$ denote the pre-image of $g_a$ under the homomorphism $K\rightarrow N$ which maps to $a$ under the homomorphism $K\rightarrow A$. It is easy to check that
$\psi_2(a+b)=\chi_a(b)g'_a\psi_2(b)$ and so $\psi_2$ is a nil-homomorphism.
\end{proof}

\begin{proposition}\label{nilhomhom} There is a homomorphism $\phi_3:A\rightarrow N_2$ such that $\phi$ composed with the factor map $N_2\rightarrow A$ is the identity map on $A$ and furthermore $\phi$ composed with the map $C(N_2)\rightarrow C(N)$ is equal to $\psi$. In particular $N_2$ is a split extension.
\end{proposition}

\begin{proof} By lemma \ref{liftnilhom} it is enough to prove that there is a homomorphism $\phi_3:A\rightarrow N_2$ such that $\phi_3$ composed with the map $N_2\rightarrow C(N_2)$ is $\psi_2$.
The function $\chi_a$ is a linear function for every $a$ and so it can be represented as an element $l_a$ in $\hat{A}\times\mathcal{C}$. Let $t_a:=l_ag'_a$. Whit this notation we have that
$\psi_2(a+b)=t_a\psi_2(b)$ for every $a,b\in A$. Now the map $\psi_3:a\rightarrow t_a$ is a homomorphism satisfying the required condition.
\end{proof}


Even the trivial nil-pattern $N=Z=\mathcal{C},|T|=1$ is interesting. As an example we give some details on that.
First of all any interpretation of $N$ is a split extension. In particular any nil representation of trivial class is a homomorphism
$$\phi:A\rightarrow A\ltimes(\hat{A}\times\mathcal{C})$$
satisfying that $\phi$ composed with the factor map $N_2\rightarrow A$ is the identity.
Assume that $\phi(t)=(t,(\chi_t,c_t))$ then we have that $\eta:t\rightarrow c_t$ is the corresponding nil-morphism. The reader can check easily $\Delta_{t_1,t_2,t_3}\eta$ is the constant $1$ function. In other words $\eta$ is a pure quadratic function on $A$. The other direction is also true. Every pure quadratic function $\psi$ on $A$ generates a quadratic nil-representation of trivial class by putting $\Delta_t\psi$ (which is the product of a character and a constant) into the second coordinate.

\medskip

\subsection{Circular functions}

\begin{definition} Let $N$ be a nil-pattern. A measurable $L_2$ function $f:C(N)\rightarrow\mathbb{C}$ is called {\bf circular} is for every $z\in Z$ and almost all $x\in C(N)$ we have $f(zx)=zf(x)$. (Recall that $Z$ acts on $C(N)$.) Let furthermore $\mathcal{H}_i(C(N)),~i\in\mathbb{Z}$ denote the set of $L_2$ measurable functions $f:C(N)\rightarrow\mathbb{C}$ such that $f(zx)=z^if(x)$. In particular $\mathcal{H}_1(C(N))$ is the set of circular functions.
\end{definition}

The next lemma is obvious.

\begin{lemma} Let $\phi:C(N)\rightarrow\mathbb{C}$ be an arbitrary circular function of absolute value $1$. Then $\mathcal{H}_i=\phi^i\mathcal{H}_0$. Furthermore every $L_2$ function $f:C(N)\rightarrow\mathbb{C}$ can be written in the form
$$f=\sum_{i\in\mathbb{Z}}\phi_i f_i$$
where $f_i\in\mathcal{H}_0(C(N))$ and
$f_i=\int_{z\in Z} f(xz)\overline{\phi_i(xz)}.$
\end{lemma}

\subsection{Nil-morphism on the ultra product group}

The basic example for an interpretation of a nil-pattern is the group $G(M)$ itself. Following the notation of the previous chapters, the group $W_0$ defines a nil-class of type $T$ and $G(M)$ is a interpretation of this class under the embedding $T\rightarrow\hat{\bA}_1$.
It follows that there is an induced map of the form $\Psi:G(M)\rightarrow C(W_0)$ where $C(W_0)$ is the core of $W_0$. Since $\bA$ is embedded into $G(M)$ we will be especially interested in the restriction $\Psi:\bA\rightarrow C(W_0)$.
The map $\Psi$ can be regarded as a nil-morphism on $\bA$.
In this chapter, our goal is to examine the properties of $\Psi$ on $G(M)$ and on $\bA$.
For this reason we start with a more explicit description that the reader can regard as an alternative definition of $\Psi$. We also check explicitly that $\Psi$ is a nil-morphism.

If $\sigma$ is an element on $G(M)$ then by lemma \ref{deltasigma} the dual support of $\Delta_\sigma\phi$ is a coset of $T$. This means that there is a linear character $\chi$ such that $(\Delta_\sigma\phi)\chi$ is in the Hilbert space $\mathcal{H}_\phi$. Let $\sigma_2$ denote element in $Q\subseteq G(M)$ corresponding to $\chi$. Then by definition $\sigma_3=\sigma\sigma_2\in W$. The element $\chi$ is unique up to a multiplication with a character from $T$ and so the left coset of $T$ containing $\sigma_3$ depends only on $\sigma$. After factoring $W$ out by $\stab(G(M),\phi)$ the image of $\sigma_3$ determines a unique element from the coset space $C(W_0)$. This element is $\Psi(\sigma)$.

Our goal is to show that $\Psi$ restricted to $\bA$ is measurable and (which is stronger then measurability) it is continuous in the $\sigma$-topology on $\bA$.
Recall that a function $f$ on $\bA$ is called continuous if it is (without any zero measure change) the ultra limit of bounded functions.

We will make use of the following lemma.

\begin{lemma}\label{charsection} Let $f:\bA\times\bA\rightarrow \mathcal{C}$ be a bounded continuous function such that for every fix $x$ the function $f(x,y)$ is measurable in $\mathcal{F}_1$ and its $L_2$ norm is positive. Then there is a continuous function $f'(x,y)$ such that for every fixed $f'(x,y)$ is a linear character for every fixed $x$ such that $(f'(x,y),f(x,y))>0$ for every $x$.
\end{lemma}

\begin{proof} The fact that $f$ is continuous is equivalent with the fact that it is the ultra limit of functions $f_i(x,y)$ on $A_i\times A_i$. For every $i$ we define the functions $f'_i(x,y)$ such that for every fix $x$ the function $f'_i(x,y)$ is one of the linear characters of $A_i$ which has maximal (in absolute value) scalar product with $f_i(x,y)$.
Let $f'(x,y)$ be the ultra limit of $\{f'_i(x,y)\}_i$. Then $|(f'(x,y),f(x,y))|$ for a fix $x$ is the ultra limit of the absolute value of the maximal fourier coefficient of $f_i(x_i,y)$ for a sequence $\{x_i\}_{i=1}^\infty$ whit ultra product $x$. This is the same as the maximal Fourier coefficient of $f(x,y)$. Using the fact the $f(x,y)$ is in $\mathcal{F}_1$ and has positive $L_2$ norm we obtain that it can't be $0$.
\end{proof}

\begin{proposition}\label{contofphi} The function $\Psi:\bA\rightarrow C(W_0)$ is continuous in the $\sigma$-topology on $\bA$.
\end{proposition}

\begin{proof} First of all by a zero measure change we can assume that $\phi$ is continuous and $\phi_i:A_i\rightarrow\mathcal{C}$ is a sequence a function whose ultra limit is $\phi$.
Now let $f_i:A_i\times A_i\rightarrow\mathcal{C}$ be the sequence of two variable functions defined by $f_i(t,x)=\Delta_t\phi_i(x)$. The ultra limit of $\{f_i\}_{i=1}^\infty$ is the function $f:\bA\times\bA\rightarrow\mathcal{C}$ defined by $f(t,x)=\Delta_t\phi(x)$ and so $f$ is continuous.
We use lemma \ref{charsection} to get a function $f':\bA\times\bA\rightarrow\mathcal{C}$ such that for every $t\in\bA$ the function $f'(t,x)$ is a character $\chi_t$ in the dual support of $\Delta_t\phi$.
Now for every $t$ we define a a unitary operator $\Xi'_t$ on $\mathcal{H}_\phi$ such that $\Xi'_t g=g^{\iota(t)}\chi_t^{-1}$. First we show that $\Xi'_t$ depend continuously on $t$ using the weak topology on operators. For this, it is enough to prove that for fixed $g,g_2\in\mathcal{H}_\phi$ the function $(\Xi'_tg,g_2)$ is continuous. This is trivial since in $(\Xi'_tg,g_2)=((\Delta_t g)\chi_t^{-1},g_2)$ everything is continuous here.

On the other hand $\Xi'_t$ is an operator in $W$ and its coset according to $T$ represents $\Psi(t)$. This completes the proof.
\end{proof}

\begin{definition} Let $\theta$ be a function in $\mathcal{F}_1$ which is equal to a linear function $\theta_2$ for almost every $t\in\bA$. Then we denote by $[\theta]$ the function $\theta_2$.
\end{definition}

The next lemma follows directly from the definition of $\Xi'_t$.

\begin{lemma}\label{skewhom2} The function $\Xi'_t$ defined in the proof of proposition \ref{contofphi} satisfies the equation $\Xi'_tT=\Psi(t)$.
Furthermore  if $t_2\in \ker(T)$ then
$$\Xi'_{t+t_2}\equiv\Xi'_t[\Delta_{t_2}\phi](t)$$ modulo the group $T$. (see Lemma \ref{kertesz} shows that $\Delta_{t_2}\phi$ is linear).
\end{lemma}

\begin{lemma} The function $\Psi$ is a nil-morphism on $\bA$.
\end{lemma}

\begin{proof} Let $a$ be an arbitrary element in $\bA$. Then using the functions $\Xi'_t$ form the proof of proposition \ref{contofphi} we have
$$\Xi'_b(\Xi'_a g)=(g^{\iota(a)}\chi_a^{-1})^{\iota(b)}\chi_b^{-1}=g^{\iota(a+b)}(\chi_a^{-1})^{\iota(b)}\chi_b^{-1}=\Xi'_{a+b}\chi_{a+b}(\chi_a^{-1})^{\iota(b)}\chi_b^{-1}.$$
This formula shows that
\begin{equation}\label{skewhom1}
\Xi'_a\Xi'_b\equiv\Xi'_{a+b}\chi_a^{-1}(b)
\end{equation}
modulo the subgroup $T\subset W_0$. In other words $\Xi'_a\Psi(b)=\Psi(b)\chi_a^{-1}(b)$.
\end{proof}

\begin{definition} A measurable function $f:\bA\rightarrow\mathbb{C}$ is called ``representable on $C(W_0)$'' if there is a Borel function $f':C(W_0)\rightarrow\mathbb{C}$ such that $f(t)=f'(\Psi(t))$ for almost every $t$.
\end{definition}

\begin{lemma}\label{charfelhuz} Every character $\chi\in T$ is representable on $C(W_0)$.
\end{lemma}

\begin{proof} The space $C(W_0)$ factored out by $Z$ is equal to $\hat{T}$.
If we represent $\chi$ on the space $\hat{T}$ in the natural way then it is easy to see that we get the desired function $\chi'$ on $C(W_0)$.
\end{proof}

\begin{corollary}\label{charfelhuz2} Every $L_2$ function in the $\sigma$-algebra of $T$ is representable on $C(W_0)$.
\end{corollary}

\begin{lemma}\label{crossection} There is a Borel measurable function $f:C(N)\rightarrow\mathcal{C}$ such that for every $c\in Z$ we have $f(xc)=cf(x)$ for almost every $x\in C(N)$.
\end{lemma}

\begin{proof} Let $S$ be a Borel section (a Borel coset representative system) for the subgroup $T\times Z$ in $N$. Then $SZ$ is a Borel section for the subgroup $T$. This means that the elements of the set $SZ$ represent the elements of $C(N)$. We define the function $f$ first on the set $SZ$ by $f(sz)=z$ where $s\in S$ and $z\in Z$. The map $f$ induces a Borel map, that we also denote by $f$, on $C(N)$ which obviously satisfy the requirement of the lemma.
\end{proof}

The next crucial lemma says that $\phi$ can be represented as Borel measurable functions on the nil-manifold $C(W_0)$.

\begin{lemma}\label{representable} The function $\phi$ is representable on $C(W_0)$.
\end{lemma}

\begin{proof} Let $\phi_2$ be the composition of $\Psi$ with the map $f$ guaranteed by lemma \ref{crossection}. Then by lemma \ref{skewhom2} we have $\phi_2(t)=f(\Xi'_tT)$ and so
if $t_2\in\ker(T)$ then by the second part of lemma \ref{skewhom2}
\begin{equation}\label{skewhom3}
\phi_2(t+t_2)=f(\Xi'_{t+t_2}T)=f(\Xi'_t[\Delta_{t_2}\phi](t) T)
\end{equation}

This implies that $\Delta_{t_2}\phi_2(t)=\Delta_{t_2}\phi(t)$ for almost every $t$.
It follows that $\phi\phi_2$ is invariant under the translations $t_2\in\ker(T)$ and so it is measurable in $T$. Now corollary \ref{charfelhuz2} finishes the proof.
\end{proof}

\begin{corollary}\label{representable2} Every function in the  $\sigma$-algebra generated by all functions in $\mathcal{H}_\phi$ is representable on $C(W_0)$.
\end{corollary}




\subsection{Reduction to finite dimension}

The Group $W_0$ is a an extension of the one dimensional group $T^*$ by the compact abelian group $\hat{T}$. However $\hat{T}$ can be infinite dimensional or it can have an infinite torsion part.  Fortunately it turns out that we can choose a different representative for the module of $\phi$
such that $T$ becomes finitely generated. This means that in this case $\hat{T}$ is a direct product with a finite dimensional torus with a finite group.

\begin{lemma}\label{finitecomp} Let $A$ be a compact topological group and $D$ be a finite dimensional subgroup in $A$. Then there is a closed subgroup $K\subseteq A$ such that $K\cap D$ is trivial and $A/K$ is finite dimensional.
\end{lemma}

\begin{proof} The fact that $D$ is finite dimensional is equivalent with the fact that the dual group of $D$ is finitely generated. The group $\hat{D}$ is a factor group in $\hat{A}$ and so there is a finitely generated group $Q\subseteq\hat{A}$ which together with $\hat{D}$ generates $\hat{A}$. This means that the kernel $K$ of the homomorphism $A\rightarrow\hat{Q}$ intersects $D$ trivially. It follows that $K$ satisfies the required conditions.
\end{proof}

\begin{lemma}\label{bilzero} Let $A_1,A_2$ be a compact abelian groups and $f:A_1\times A_2\rightarrow \mathcal{C}$ a continuous bilinear function. Then there is a closed finite index subgroup $B\subseteq A_1$ such that $f$ is trivial on $B\times A_2$.
\end{lemma}

\begin{proof} For every fixed $x\in A_1$ the map $f_x(y)=f(a,y)$ is a continuous homomorphism from $A$ to $\mathcal{C}$ and so it is an element in the dual group $\hat{A_2}$. Since the dual group is discrete and $f_x$ depend continuously on $x$ the compactness of $A_1$ implies that there are only finitely many different functions of the form $f_x$. It follows that there is a finite index subgroup $B\subseteq A_1$ such that $f_x$ is trivial for $x\in B$.
\end{proof}

\begin{lemma}\label{locbizero} Let $A$ be a locally compact abelian group, $H\subseteq A$ be a discrete co-compact subgroup and $(-,-)$ be a continuous bilinear form $A\times A\rightarrow\mathcal{C}$.
Then there is an open subgroup $B\subseteq A$ and finite co-dimensional compact subgroup $C$ in $B$ such that $BH=A$ and $(C,B)=0$.
\end{lemma}

\begin{proof} The structure theorem of locally compact abelian groups says that $A$ has an open subgroup $O$ of the form $O=O_1\times O_2$ such that $O_1$ is compact and $O_2$ is isomorphic with $\mathbb{R}^d$ for some natural number $d$. The fact that $H\cap O_1$ is finite implies that $(H\cap O)/O_1$ is a discrete subgroup in $\mathbb{R}^d$ and so $H\cap O$ is finitely generated.
Let $h_1,h_2,\dots,h_r$ be a generating system of $H\cap O$.

The fact that $O$ is open implies that $HO$ is open and since $A/H$ is compact we get that $HO$ has finite index. Let $t_1,t_2,\dots,t_n$ be a coset representative system for $HO$ in $A$.

Let $O_3$ be the subgroup of $O_1$ defined by
$$O_3=\{x~|~x\in O_1,~(x,t_i)=0,~(x,h_j)=0~~\forall~1\leq i\leq n,~1\leq j\leq r\}.$$
Clearly, $O_3$ is finite (at most $n+r$)co-dimensional in $O_1$.

The fact that $(O_3,h_i)=0$ (and thus $(O_3,O\cap H)=0$) implies that the bilinear form on $O_3\times O$ factors through the map $O\rightarrow O/(O\cap H)$. Since $O/(O\cap H)$ is compact we obtain by lemma \ref{bilzero} that there is a subgroup $C$ in $O_3$ such that $(C,O)=0$ and $C$ is finite co-dimensional in $O_3$.

Let $B$ be the subgroup $\{x~|~x\in A, (C,x)=0\}$. We have that $B$ contains $O$ and the elements $\{t_1,t_2,\dots,t_n\}$. It follows that $BH=A$.
\end{proof}

\begin{theorem}\label{veges} Every quadratic nil-pattern $0\rightarrow T\times Z\rightarrow N\rightarrow\hat{T}\rightarrow 0$ is the re-interpretation of a finite dimensional e.g. finitely generated nil-pattern of type $T_2$ under a monomorphism $T_2\rightarrow T$.
\end{theorem}

\begin{proof}The group $N/Z$ is a locally compact abelian group with co-compact discrete subgroup $H=T/Z\equiv T$. The map $(x,y):=[x,y]$ is a continuous bilinear map on $N/Z\times N/Z$ to $\mathcal{C}$. Let $C\subseteq B\subseteq N/Z$ be the subgroups guaranteed by lemma \ref{locbizero}. Let furthermore $C_2$ and $B_2$ be the pre-images of $B$ and $C$ under the homomorphism $N\rightarrow N/Z$.

The groups $C_2\subseteq B_2$ satisfy that $[C_2,B_2]$ is trivial and so $B_2$ is contained in the centralizer of $C_2$. In particular $C_2$ is abelian and compact. By lemma \ref{finitecomp} we have a finite co-dimensional compact subgroup $C_3$ in $C_2$ such that $C_3\cap (Z\times T)$ is trivial.

Let $K$ denote the centralizer of $C_3$ in $N$. Let $T_2$ denote the intersection of $T$ with $K$.
The fact that $C_3$ is finite co-dimensional implies that $T_2$ is finitely generated.
The fact that $K(Z\times T)=N$ implies that $K/C_3$ is a quadratic nil-pattern of type $T_2$.

It can be seen from the construction of reinterpretations that $N$ is a re-interpretation of the nil-pattern $K/C_3$ under the monomorphism $T_2\rightarrow T$.
\end{proof}

The next theorem is a very important consequence of the previous one.

\begin{theorem}\label{fingentype} Every rank one module $M$ in $\hat{\bA}_2$ contains a character $\phi$ such that the type of $\phi$ is finitely generated.
\end{theorem}

\begin{proof} Let $\phi'$ be an arbitrary character in $M$. Then by lemma \ref{representable} we know that $\phi'$ is representable on $C(W_0)$ for some nil-morphism $\Psi:\bA\rightarrow C(W_0)$ where $W_0$ is of type $T$. This means that $\phi'=\Psi\circ\phi'_2$ for some circular function $\phi'_2$. Theorem \ref{veges} implies that there is a finitely generated abelian group $T_2$ embedded into $T$ and a nil-pattern of type $T_2$ such that $C(W_0)$ is the interpretation of $N$ over the embedding $T_2\rightarrow T$. This generates an epimorphism $h:C(W_0)\rightarrow C(N)$. If $\phi_2:C(N)\rightarrow\mathbb{C}$ is circular of absolute value one then the composition $\phi_3=h\circ\phi_2$ is again circular. Since $\phi'_2$ is circular it means that $\phi'_2=\phi_3 f$ for some function $f$ in $\mathcal{H}_0(C(W_0))$. It follows that the composition $\phi=\Psi\circ\phi_3$ is a function such that $\phi'=\phi \Psi\circ f$. Moreover $\Psi\circ f$ is obviously measurable in $\mathcal{F}_1$.
\end{proof}

\subsection{Almost homomorphisms}

Before stating the main theorem in this chapter we need some notation.
First of all we will use the metric $d$ on the circle group $\mathcal{C}$ defined in a way that $d(x,y)$ is the length of the shorter arch on $\mathcal{C}$ connecting $x,y$ divided by $2\pi$.
We extend this metric to groups $N$ that are central extensions of $\mathcal{C}$.
If $x,y$ are in a different coset of $\mathcal{C}$ then $d(x,y)=\infty$. If $x,y$ are contained in the same coset $r\mathcal{C}$ then $d(xr,yr)=d(x,y)$. In other words $d$ is shift invariant on $N$.
Using the fact that $N$ is a central extension, there is a metric (which is unique) satisfying this equation.

Let $X\rightarrow\mathcal{C}$ be a random variable. We say that $X$ is concentrated if there is a short arch of length at most $2\pi/3$ containing all the values.
The significance of concentrated random variables is that they have a well defined expected value which is between the two extremal points of the range of $X$ on the short arch.

This expected value can be computed by considering a universal cover map $p:\mathbb{R}\rightarrow\mathcal{C}$. There is a lift (defined up to a shift in $\mathbb{Z}$) of $X$ to another variable $X'\rightarrow\mathbb{R}$ such that the range of $X'$ is contained in an interval of length $1/3$. The expected value of $X'$ composed with $p$ is a unique value.
We also have that if $X_1,X_2$ are both concentrated on an arch of length $2\pi/6$ the $\mathbb{E}(X_1X_2)=\mathbb{E}(X_1)\mathbb{E}(X_2)$.

\begin{theorem}\label{almosthom} Let $A$ and $B$ be two finite groups and let $N$ be a central extension of the circle group $\mathcal{C}$ by $B$. Assume that $h:A\rightarrow N$ is a map such that
$h\rightarrow N/\mathcal{C}$ is a homomorphism and
$$d(h(a+b),h(a)h(b))\leq\epsilon\leq 1/40$$ for every $a,b\in A$. Then there is a homomorphism $g:A\rightarrow N$ such that $g(a)\mathcal{C}=h(a)\mathcal{C}$ and $d(g(a),h(a))\leq4\epsilon$ for every $a\in A$.
\end{theorem}

\begin{proof} First of all it is trivial that there is a function $h'$ satisfying the formula
$$(h'(a)=h(a)\vee h'(-a)=h(-a))\wedge (h'(a)^{-1}=h(-a))$$
for every $a\in A$. We have that $d(h'(a),h(a))\leq\epsilon$ because $h'(a)$ is either $h(a)$ or $h(-a)^{-1}$ and in the second case $$d(h'(a),h(a))=d(h'(a)h(-a),h(a)h(-a))=d(h(0),h(a)h(-a))\leq\epsilon.$$
It follows that
\begin{equation}\label{3epsilon}
d(h'(a+b),h'(a)h'(b))\leq 3\epsilon.
\end{equation}
Now we define
$$g(a)=\mathbb{E}(h'(a_1)h'(a_2)|a_1+a_2=a)$$
where $a_1$ (or equivalently $a_2$) is chosen uniformly at random.
To show that $g(a)$ is well defined we need that the random variable $h'(a_1)h'(a_2)$ is concentrated. This follows from (\ref{3epsilon}) since there is an arch of length at most $6\epsilon$ containing the range. We have that $d(g(a),h'(a))\leq 3\epsilon$ and $d(g(a),h(a))\leq 4\epsilon$.
Now to see that $g$ is a homomorphism we write
$$g(c_1+c_2)=\mathbb{E}(h'(a_1+c_2)h'(a_2)|a_1+a_2=c_1)=\mathbb{E}(h'(a_1+c_2)h'(-a_1)h'(a_1)h'(a_2)|a_1+a_2=c_1).$$
Now we introduce the random variables $X_1=h'(a_1+c_2)h'(-a_1)$ and $X_2=h'(a_1)h'(a_2)$. Both random variables are concentrated on an arch of length $6\epsilon$ on a fix circle in $N$. It follows that $g(c_1+c_2)=g(c_1)g(c_2)$ since $\mathbb{E}(X_1)=g(c_1)$ and $\mathbb{E}(X_2)=g(c_2)$.
\end{proof}

\subsection{Constructing nil-morphisms on finite groups}

If $N$ is a finitely generated nil-pattern of type $T,\mathcal{C}$ then $C(N)$ is a smooth compact manifold which has an embedding $p$ into a finite dimensional Euclidean space. The embedding $p$ induces a metric $\delta$.

\begin{definition} Let $N$ be a finitely generated nil pattern and $A$ an abelian group. A function $f:A\rightarrow C(N)$ is called an $\epsilon$ almost nil-morphism if for every element $a\in A$ there is an element $g$ from $N$ and a linear function (a linear character multiplied by a constant in $\mathcal{C}$) $\chi$ such that for every element $b\in A$ we have $\delta(\psi(a+b),\chi(b)g\psi(b))\leq\epsilon$.
\end{definition}

We will also need an Abelian version of the above definition.

\begin{definition} Let $T$ be a finitely generated abelian group. We say that $f:A\rightarrow\hat{T}$ is an $\epsilon$-linear function if for every $a\in A$ there is an element $c\in\hat{T}$ such that for every $d_2(f(a+b)-f(b),c)\leq\epsilon$ for every $b\in A$ (where $d_2$ is the Riemann metric on $\hat{T}$).
\end{definition}

\begin{lemma}\label{almosthom2} Let us fix the finitely generated group $T$. For every $\epsilon>0$ there is a number $\delta>0$ such that every $\delta$-liner function can be modified to a linear function with error $\epsilon$.
\end{lemma}

\begin{proof} The group $\hat{T}$ is the direct product of a torus with a finite abelian group. It is enough to prove the statement for toruses. A torus is a direct power of the circle group. By doing the approximation coordinate wise we can reduce the problem to $\hat{T}=\mathcal{C}$ which follows from theorem \ref{almosthom}
\end{proof}

\begin{lemma}\label{almostnilhom} For every nil-pattern $N$ there is a positive constant $\epsilon>0$ there is a constant $\delta>0$ such that if $f$ is a $\delta$ almost nil-morphism on a finite abelian group $A$ then there is a nil character $f':A\rightarrow C(N)$ such that $\delta(f'(a),f(a))\leq\epsilon$.
\end{lemma}

\begin{proof} Let $\tau$ denote the homomorphism $N\rightarrow\hat{T}$. With this assumption, if $f$ is a $\delta_2$ almost nil-morphism then $f$ composed with the map $\tau$ is an almost linear function into the finite dimensional group $\hat{T}$. This means by lemma \ref{almosthom2} that if $\delta_2$ is small enough then $f$ can be modified with an error $\epsilon_2$ in the metric $\delta$ such that the new function $f'$ composed with $\tau$ is a linear function. By applying an appropriate translation from $N$ (with changing the errors only in a constant way) we can also assume that $f'$ composed with $\tau$ is a homomorphism. This means that there is a finite subgroup $H$ in $\hat{T}$ with dual group $T_3=T/T_2$ (for some finite index subgroup $T_2$ in $T$) such that $f'$ composed with $\tau$ maps $A$ surjectively on $H$. Let $H'$ denote the group $\tau^{-1}(H)$. The group $T_2$ is in the center of $H'$ and thus the restriction of the map $N\rightarrow C(N)$ is well defined on the group $H_2=H'/T_2$. We can interpret $H_2$ as a nil-pattern of type $T_3$ and $f'$ as a map going to $C(H_2)$. Now it is easy to see that $f':A\rightarrow C(H_2)$ is also an almost nil-morphism by modifying the elements $g$ (see the definition of almost-nil characters) to preserve the group $H_2$.
Let $N_2$ denote the interpretation of $H_2$ under the dual map of $\tau$ and let $f_2$ be the lift of $f'$ to $C(N_2)$.
Similarly as in proposition \ref{nilhomhom} we see that there is an almost homomorphism $\phi_3:A\rightarrow N_2$ whose composition with the map $N_2\rightarrow C(N_2)$ is suitably close to $f_2$. Then by theorem \ref{almosthom} we can correct $\phi_3$ to a homomorphism $\phi'_3$ with a small error. This completes the proof.
\end{proof}

\begin{theorem}\label{rigidity} There is a sequence of nil-morphisms $\Psi_i:A_i\rightarrow C(W_0)$ such that their ultra limit is $\Psi$.
\end{theorem}

\begin{proof} We start with a sequence of functions $\Psi'_i:A_i\rightarrow C(W_0)$ such that the ultra limit of $\Psi'_i$ is $\Psi$. It is easy to see that there is a sequence $\epsilon_i$ tending to $0$ such that $\Psi'_i$ is an $\epsilon_i$ almost nil-morphism when $i\in I$ for an index set $I$ in the ultra filter $\omega$. Then lemma \ref{almostnilhom} completes the proof.
\end{proof}

\subsection{Proof of the quadratic structure theorem}

First of all note that it is enough to prove a seemingly weaker statement in which $M$ also depend on $f$. The reason is that there are only finitely many nil-patterns of dimension less then $n$.
Taking their direct product a universal one can be constructed.

We show the theorem by contradiction and assume that there is a number $\epsilon$ and function $F$ not satisfying the conditions of the theorem. This means that there is a sequence of growing abelian groups $\{A_i\}_{i=1}^\infty$ and functions $f_i:A_i\rightarrow \mathbb{C}$ such that $f_i$ doesn't satisfy the condition with $n$.

Let $f$ be the ultra limit of the sequence $\{f_i\}$ and let
$f=g+\phi_1+\phi_2+\dots$ be its quadratic decomposition where $\|g\|_{U_3}=0$ and $\|g\|_\infty\leq 1$.
Let $n_1$ be a natural number such that $$\|f-(g+\sum_{i=1}^{n_1}\phi_i)\|_2\leq\epsilon/5$$.
From theorem \ref{fingentype} we get that there are functions $\phi'_i$ for every $i$ such that $\phi'_i$ is in the same rank one module as $\phi_i$, the type of $\phi'_i$ is finitely generated, and $\|\phi_i-\phi'_i\|\leq\epsilon/(5n_0)$. We can also assume that $(\phi_i,\phi_i-\phi'_i)=0$.

According to corollary \label{representable2} there are nil-morphisms $\Psi_i:\bA\rightarrow C(N_i)$ for some nil-patterns $N_i$ such that $\phi'_i$ is representable on $C(N_i)$.
Consider the manifolds $C(N_i)$ already embedded into finite dimensional spaces. Since polynomial functions are dense in $L_2(C(N_i))$ it follows that there are polynomials $p_i:C(N_i)\rightarrow\mathbb{C}$ such that $\|\phi'-p_i(\Psi_i)\|_2\leq \epsilon/(5n_0)$.

For every $i$ let $\{\Psi_i^j\}_{j=1}^\infty$ be a sequence of nil-morphisms $A_i\rightarrow C(N_i)$ guaranteed by theorem \ref{rigidity} converging to $\Psi_i$.
Let $\{g_j\}_{j=1}^\infty$ be a sequence of functions converging to $g$.
Then $f_j=g_j+\sum_{i=1}^{n_0}p_i(\Psi_i^j)+h_j$ where $h_j$ is some error term such that $\limsup\|h_j\|_2\leq\epsilon/2$. Let $d_j=\sum_{i=1}^{n_0}p_i(\Psi_i^j)$. We also have that $\limsup (d_j,h_j)=\limsup (d_j,g_j)=\limsup (h_j,g_j)=0$. Let $n$ be the maximum of $n_0$, the degrees of $\{p_i\}_{i=1}^{n_0}$ and the dimensions of $N_i$.
We obtain that every $f_i$ satisfies the condition of the theorem with $n$ for infinitely many indices which is a contradiction.

\subsection{Some other inverse theorems}

Nil-morphisms, on an abelian group $A$, take values in a nil-manifold $N/T$ and so we can't take the scalar product of an ordinary function $f:A\rightarrow\mathbb{C}$ with a nil-morphism.
In this part we define two alternative correlation notions that are suitable for inverse theorems.

\begin{definition} Let $\{f_i:C(N)\rightarrow\mathbb{C}\}_{i=1}^\infty$ be a fixed (arbitrary) system of continuous functions such that the $L_2$-span $\langle{f_i\}_{i=1}^\infty}\rangle$ contains every measurable function $f:C(N)\rightarrow\mathbb{C}$. We say that a nil-morphism $\phi:A\rightarrow C(N)$ $\delta$-correlates with a function $g:A\rightarrow\mathbb{C}$ if $(g,f_i\circ\phi)\leq\delta$ for some $i\leq 1/\delta$.
\end{definition}

\begin{remark} Note that it is easy to construct function sequences of the above type. If we fix any embedding of $C(N)$ into a finite dimensional Euclidean space then polynomials of the coordinate functions (with rational coefficients) have the required property.
\end{remark}

It turns out rather surprisingly that the condition on the set $\{f_i\}$ is way too strong for applications. Inverse theorems can be formulated using much simpler systems.

\begin{definition} Let $\{f_i:C(N)\rightarrow\mathbb{C}\}_{i=1}^\infty$ be a fixed (arbitrary) system of continuous functions such that $\lim_{i\rightarrow\infty}f_i$ is a circular function of absolute value $1$.
We say that a nil-morphism $\phi:A\rightarrow C(N)$ $\delta$-correlates with a function $g:A\rightarrow\mathbb{C}$ if $(g,\chi (f_i\circ\phi))\leq\delta$ for some $i\leq 1/\delta$ and linear function $\chi:A\rightarrow\mathbb{C}$.
\end{definition}

\end{document}